\documentclass[a4paper]{article}

\usepackage[top=3cm,bottom=3cm,left=3cm,right=3cm]{geometry}

\usepackage{palatino}
\usepackage{latexsym}
\usepackage{bbold} 
\usepackage[mathscr]{eucal}
\usepackage{amsmath}
\usepackage{amsthm}
\usepackage{amsfonts}
\usepackage{amssymb}
\usepackage{amscd}
\usepackage{color}
\usepackage{graphicx}
\usepackage{graphics}
\usepackage{pifont}
\usepackage[makeroom]{cancel}
\usepackage[normalem]{ulem}
\usepackage[dvipsnames]{xcolor}
\usepackage{tikz}
\usepackage{cite}
\usepackage{cleveref} 
\usepackage{enumitem} 

\usepackage{mathtools} 

\theoremstyle{plain}
\newtheorem{theorem}{Theorem}
\newtheorem{remark}[theorem]{Remark}
\newtheorem{lemma}[theorem]{Lemma}
\newtheorem{proposition}[theorem]{Proposition}
\newtheorem{definition}[theorem]{Definition}

\newcommand{\T}{\mathbb{T}}
\newcommand{\R}{\mathbb{R}}
\newcommand{\N}{\mathbb{N}}

\newcommand{\kb}[1]{{\color{ForestGreen}(\textbf{KB: }#1)}}

\title{Approximate controllability of a bilinear wave equation and minimum time}

\begin{document}
\author{Karine Beauchard\footnote{Univ Rennes, CNRS, IRMAR - UMR 6625, F-35000 Rennes, France (karine.beauchard@ens-rennes.fr)},\quad Thomas Perrin\footnote{Univ Rennes, CNRS, IRMAR - UMR 6625, F-35000 Rennes, France (thomas.perrin@ens-rennes.fr)},\quad Eugenio Pozzoli\footnote{Univ Rennes, CNRS, IRMAR - UMR 6625, F-35000 Rennes, France (eugenio.pozzoli@univ-rennes.fr)}}
\maketitle

\abstract{We study the global approximate controllability (GAC) of a Klein-Gordon wave equation, posed on the torus $\mathbb{T}^d$ of arbitrary dimension $d\in \mathbb{N}^*$, with bilinear control potentials supported on the first $(2d+1)$-Fourier modes. Let $Z(W_0)\subset \mathbb{T}^d$ be the set of essential zeroes of the initial state $W_0\in H^1\times L^2(\mathbb{T}^d)$, and $r(W_0)\geq 0$ be the maximum radius of a ball of $\mathbb{T}^d$ contained in $Z(W_0)$. Due to finite speed of propagation, the minimum control time starting from $W_0$ is necessarily larger than or equal to $r(W_0)$. We prove the following three facts. 

In low dimensions $d \in \{1,2\}$: the minimum time for GAC from $W_0 \neq 0$ is equal to $r(W_0)$. 

In any dimensions $d\geq 3$: the minimum time for GAC from $W_0$ is zero if $Z(W_0)$ has zero Lebesgue measure; and the GAC in sufficiently large time from all $W_0\neq 0$.

The proof strategy consists in combining Lie bracket techniques \emph{à la Agrachev-Sarychev} with the propagation of well-prepared positive states.
}




\section{Introduction}

\subsection{The system}

We consider the Klein-Gordon wave equation controlled with a potential, posed on the multi-dimensional torus $\T^d:=\R^d/2\pi \mathbb{Z}^d$, 
\begin{equation}\label{eq:wave_bis}
{\partial^2_t}w(t,x)=\left(\Delta-1+V(x)+\sum_{j=0}^{2d}u_j(t)V_j(x)\right)w(t,x),\quad (t,x)\in \mathbb{R} \times \mathbb{T}^d,
\end{equation}
where 
$d \in \mathbb{N}^*$,
$u=(u_0,\dots,u_{2d}):\mathbb{R} \to \mathbb{R}^{2d+1}$,
\begin{equation}\label{eq:cos-sin_bis}
V_0(x)=1, \quad V_{2j-1}(x)=\sin(x_j), \quad V_{2j}(x)=\cos(x_j) \quad \text{ for } j=1,\dots,d,
\end{equation}
and $V \in L^{\rho}(\T^d)$ where
\begin{equation} \label{def:q_rho}
\rho = \left\lbrace \begin{array}{ll}
2 & \text{ if } d=1, \\
2^+ & \text{ if } d=2, \\
d & \text{ if } d \geq 3.
\end{array}\right.
\end{equation}
These assumptions are valid in the whole article.
We introduce its vectorised version, for $W=(w,\partial_t w)$,
\begin{equation} \label{wave_vect}
\frac{dW}{dt}(t)=\left( A+VB+\sum_{j=0}^{2d} u_j(t) V_j B \right) W(t)
\end{equation}
where
\begin{equation} \label{def:B}
A=\begin{pmatrix}
    0 & 1 \\ \Delta-1 & 0
\end{pmatrix},
\qquad
B=\begin{pmatrix}
    0 & 0 \\ 1 & 0
\end{pmatrix}.
\end{equation}
Then, the following well-posedness result holds (see \Cref{sec:WP} for a proof).

\begin{proposition} \label{Prop:WP}
For every $T>0$, $u:[0,T] \to \R^{2d+1}$ piecewise constant and $W_0 \in H^1 \times L^2(\T^d)$ there exists a unique solution 
$W \in C^0([0,T],H^1 \times L^2(\T^d)) \cap C^1([0,T],L^2 \times H^{-1}(\T^d))$ of \eqref{wave_vect} associated with the initial condition
$W(0)=W_0$. It will be denoted $W(t;u,W_0)$.   
\end{proposition}

System \eqref{wave_vect} is a control system, with state $W(t) \in H^1 \times L^2(\T^d)$ and control $u:[0,T] \to \R^{2d+1}$ piecewise constant. The control potentials $V_0B,\dots,V_mB$ define bounded operators on $H^{s+1}\times H^{s}(\T^d)$ for any $s\geq 0$, therefore exact controllability in those spaces is impossible \cite{BMS}. We study the approximate controllability.


\begin{definition}
Let $W_0 \in H^1 \times L^2(\T^d)$ and $T \in [0,\infty)$. 
\begin{itemize}[topsep=0cm, parsep=0cm, itemsep=0cm]
\item A target $W_f=(w_f,\dot{w}_f) \in H^1 \times L^2(\T^d)$ is $H^1 \times L^2$-approximately reachable in time $T^+$ from $W_0$ if, for every $\epsilon>0$, there exist $T_1 \in [T,T+\epsilon]$ and $u \in PWC([0,T_1],\R^{2d+1})$ such that $\|W(T_1;u,W_0)-W_f\|_{H^1 \times L^2}<\epsilon$.
\item System \eqref{wave_vect} is $H^1 \times L^2$-approximately controllable in time $T^+$ from $W_0$ if any target in $H^1 \times L^2(\T^d)$ is $H^1 \times L^2$-approximately reachable in time $T^+$ from $W_0$. 
\item The minimum time for $H^1 \times L^2$-approximate controllability from $W_0$ is 
\begin{equation}\label{eq:equality-rays}
\mathcal{T}(W_0)\!=\!\inf\!\left\{ T\in \R_+ ; \text{ System } \eqref{wave_vect} \text{ is } H^1 \times L^2 \text{-approximately controllable in time } T^+ \text{ from } W_0 \right\}
\end{equation}
(it is $+\infty$ when System \eqref{wave_vect} is not $H^1 \times L^2$-approximately controllable from $W_0$).
\item System \eqref{wave_vect} is small-time $H^1 \times L^2$-approximately controllable from $W_0$ if $\mathcal{T}(W_0)=0$.
\end{itemize}
\end{definition}

\begin{definition}[Notations $Z$, $r$]
For $n \in \mathbb{N}^*$ and $f \in L^{2}(\T^d,\R^n)$, 
\begin{itemize}[topsep=0cm, parsep=0cm, itemsep=0cm]
\item $Z(f)\subset\T^d$ denotes the set of zeros of $f$, i.e.  $Z(f)=f^{-1}(\{0\})$,
\item $r(f)\in \R_+$ denotes the maximum radius of a ball of $\T^d$ contained in $Z(f)$\footnote{The set $Z(f)$ is defined up to a set of zero Lebesgue measure. For simplicity, we use the same vocabulary for $Z(f)$ as for perfectly well-defined sets. For example, a correct definition of $r(f)$ would be $r(f)=\sup \{ \widetilde{r} \geq 0 ; f=0$ a.e. on a ball of $\T^d$ of radius  $\widetilde{r} \}$. Similarly, we write ``$Z(f)=Z(g)$'' when $f=0$ a.e. on $Z(g)$ and $g=0$ a.e. on $Z(f)$. We also say ``$Z(f)$ is closed'' when there exists a representative $\widetilde{f}$ of the equivalence class of $f$
(i.e. $\widetilde{f}$ is a function defined at every point $x \in \T^d$) such that $\widetilde{f}^{-1}(\{0\})$ is closed.}.
\end{itemize}
\end{definition}

Due to the finite propagation speed of waves,  we have the following lower bound on the minimum time (the proof is classical, recalled for completeness in \Cref{sec:propagation}).

\begin{proposition} \label{Prop:propagation}
For any $W_0 \in H^1 \times L^2(\T^d)$, we have $\mathcal{T}(W_0)\geq r(W_0)$.
\end{proposition}

Thus, in the bilinear control of the wave equation, the support of the initial condition plays the same role as the support of the source term in the linear control problem. Note that, when $Z(f)$ is closed, then $r(f)$ is the minimum time such that, for any point $x\in Z(f)$, there exists a geodesic of $\T^d$ starting at $x$ reaching the open set $\T^d \setminus Z(f)$ in time smaller than $r(f)$: this formulation intervenes in the geometric control conditions (see the seminal works of Bardos--Lebeau--Rauch and Burq--Gérard \cite{BLR, BurqGerard97}).

\medskip

In particular, \Cref{Prop:propagation} proves the implication 
\begin{equation} \label{implic}
\text{ Small-time } H^1 \times L^2\text{-approximate controllability from } W_0
\quad \Rightarrow \quad
\text{Int}\left( Z(W_0) \right) = \emptyset.
\end{equation}
Thus, natural open problems are the following ones.
\begin{itemize}[topsep=0cm, parsep=0cm, itemsep=0cm]
\item \emph{Question 1:} Is the reciprocal of \eqref{implic} true?
\item \emph{Question 2:} Is $\mathcal{T}(W_0)=r(W_0)$ for any $W_0 \in H^1 \times L^2(\T^d)$?
\end{itemize}

\subsection{Main results}

Our main results are summarized in the following statement, where $|\cdot|$ denotes the Lebesgue measure on $\T^d$.
\begin{theorem} \label{Main}
Let  $W_0 \in H^1 \times H^1 (\T^d)$ be a non zero initial condition.
\begin{enumerate}
    \item If $|Z(W_0)|=0$, then $\mathcal{T}(W_0)=0$.
    \item We assume $V=0$.
    \begin{enumerate}
        \item If $d \in \{1,2\}$, then $\mathcal{T}(W_0)=r(W_0)$.
        \item If $d=3$, then $\mathcal{T}(W_0)<\infty$.
        \item If $d \geq 4$ and $Z(W_0)$ is closed, then $\mathcal{T}(W_0)<\infty$.
    \end{enumerate}
\end{enumerate}
\end{theorem}

The regularity assumption on the initial condition $W_0 \in H^1\times H^1$ is stronger than the one expected $W_0 \in H^1\times L^2$ for technical reasons. Indeed, in our control strategy, the first step consists in adding the velocity into the profile (when the profile is vanishing on a region where the velocity is not) and this operation would not make sense if the velocity was less regular than the profile.

Statement \emph{1.} is a partial answer to \emph{Question 1}: the hypothesis $|Z(W_0)|=0$ implies\newline$\text{Int}\left( Z(W_0) \right) = \emptyset$ but it is strictly stronger. 

Statement \emph{2.(a)} gives the reciprocal of \eqref{implic} and answers \emph{Question 2}, in the low-dimensional case $d\in\{1,2\}$, for $V=0$. This is the first example of bilinear control for the wave equation where the minimal time for global approximate controllability can be determined.

Statements \emph{2.(b)} and \emph{2.(c)} correspond to the large time $H^1 \times L^2$-approximate controllability from $W_0$. Statement  \emph{2.(c)} applies in particular when $W_0$ is a continuous function on $\T^d$, which holds, for instance, when $W_0 \in H^{s} \times H^s (\T^d)$ with $s>\frac{d}{2}$.
We refer to \Cref{subsec:Z(W0)close} for a discussion about the closedness assumption on $Z(W_0)$.


\subsection{Detailed results}

Our first result concerns the small time controllability of the velocity.

\begin{theorem} \label{Prop:STAC_vel}
Let $W_0=(w_0,\dot{w}_0)\in H^1 \times  L^2(\T^d)$ and $\dot{w}_f \in L^2(\T^d)$. We assume that $\dot{w}_f-\dot{w}_0$ vanishes a.e. on $Z(w_0)$. Then the target $(w_0,\dot{w}_f)$ is small-time 
$H^1 \times L^2$-approximately reachable from $W_0$.
In particular, if $|Z(w_0)|=0$ then any target in $\{w_0\} \times  L^2(\T^d)$ is small-time $H^1 \times L^2$-approximately reachable from $W_0$.


\end{theorem}

Our second result is the determination of some initial conditions from which System \eqref{wave_vect} is small-time $H^1 \times L^2$-approximately controllable. The last statement of \Cref{Prop:STAC} does not imply the first one; it corresponds to \Cref{Main}\emph{.1.}

\begin{theorem} \label{Prop:STAC}
$\mathcal{T}(W_0)=0$ for any initial condition
\begin{itemize}[topsep=0cm, parsep=0cm, itemsep=0cm]
\item $W_0=(w_0,\dot{w}_0) \in H^1 \times L^2(\T^d)$ such that $|Z(w_0)|=0$,
\item $W_0=(w_0,\dot{w}_0) \in H^1 \times H^1(\T^d)$ such that $|Z(W_0)|=0$.
\end{itemize}
\end{theorem}

Our third result concerns the particular case $d \in \{1,2\}$ and $V=0$. The last statement of \Cref{Prop:Tmin} corresponds to \Cref{Main}\emph{.2.(a).}

\begin{theorem} \label{Prop:Tmin}
We assume $d \in \{1,2\}$ and $V =0$.  
Then $\mathcal{T}(W_0)=r(W_0)$ for any non zero initial condition
\begin{itemize}[topsep=0cm, parsep=0cm, itemsep=0cm]
\item $W_0=(0,\dot{w}_0) \in H^1 \times L^2(\T^d)$ such that 
$\dot{w}_0$ is $\geq 0$  (resp $\leq 0$) a.e. on $\T^d$.
\item $W_0=(w_0,\dot{w}_0) \in H^1 \times L^2(\T^d)$ such that $Z(w_0)=Z(W_0)$.
\item $W_0=(w_0,\dot{w}_0) \in H^1 \times H^1(\T^d)$.
\end{itemize}
\end{theorem}

Our fourth result concerns the $H^1 \times L^2$- approximate controllability in large time. The last two statements of \Cref{Prop:large_time} corresponds to \Cref{Main}\emph{.2.(b)} and \emph{.2.(c)} 

\begin{theorem} \label{Prop:large_time}
We assume $d \geq 3$ and $V=0$. Then $\mathcal{T}(W_0)$ is finite for any non zero initial condition
\begin{itemize}[topsep=0cm, parsep=0cm, itemsep=0cm]
\item $W_0=(0,\dot{w}_0)$ where 
$\dot{w}_0 \in H^{\frac{d}{2}-1+}(\T^d)$ is $\geq 0$ (resp. $\leq 0$) a.e. on $\T^d$,
\item $W_0=(w_0,\dot{w}_0) \in H^1 \times L^2(\T^d)$ such that $Z(w_0)=Z(W_0)$ is closed,
\item $W_0=(w_0,\dot{w}_0) \in H^1 \times H^1(\T^d)$ such that $Z(W_0)$ is closed.
\end{itemize}
When $d=3$, the last 2 statements hold without the closedness assumption on $Z(W_0)$.
\end{theorem}

\subsection{Proof strategy}

\subsubsection{Approximating flows along commutators}

A fundamental idea in control theory is the following:
when trajectories can approximate the flows generated by certain operators (as $A_V:=A+VB$ and $V_j B$) then they could also approximate those of other operators, given by (well defined) elements of ${\rm Lie}\{A_V,V_jB\}$.
In this article, we use the following commutators: for $f \in L^{\infty}(\T^d)$,
\begin{equation}\label{eq:commutator1}
[f B,[f B,A_V]]=-2f^2B,
\end{equation}
\begin{equation}\label{eq:commutator2}
\sum_{j=0}^\infty
\frac{1}{j!}
{\rm ad}^j_{\log(t)F}(tA_V) \xrightarrow[t\to 0] {} B^*, \qquad F:=[B,A_V]=BA_V-A_VB.
\end{equation}
Here, $B^*$ denotes the transpose matrix of $B$ (see \eqref{def:B}). For example, when $f$ is one of the control potentials $V_j$, then $f^2$ is linearly independent from $f$, thus  \eqref{eq:commutator1} provides a new generator.
Iterating this procedure allows to 
approximate $\exp(\phi B)$ for any function $\phi$ in $L^\rho(\T^d)$
\begin{equation}\label{eq:flow-of-B}
e^{\phi B}\begin{pmatrix}
    w_0 \\ \dot{w}_0
\end{pmatrix}
=
 \begin{pmatrix}
    w_0 \\ \dot{w}_0+\phi w_0
\end{pmatrix}.
\end{equation}
In this way, we realize the instantaneous transitions
\begin{equation} \label{exp(phiB)_target}
    \begin{pmatrix}
    w_0 \\ \dot{w}_0
\end{pmatrix}
\quad \xrightarrow[]{\quad e^{\phi B} \quad }  \quad
 \begin{pmatrix}
    w_0 \\ \dot{w}_0+\phi w_0
\end{pmatrix}
\end{equation}
that change the velocity while maintaining the profile. This argument is the main idea behind the control of the velocity, stated in Theorem \ref{Prop:STAC_vel}.
The constraint on the zeroes of the velocity $Z(w_0) \subset Z(\dot{w}_f-\dot{w}_0)$ appears naturally, as the profile $w_0$ multiplies $\phi$. For more details, we refer to the Appendix \ref{appendix:brackets}, Propositions \ref{Cor:strongCV} and \ref{Prop:H1*H1}.

\medskip 

Applying instead the flow generated by $B^*$, obtained in \eqref{eq:commutator2}, gives 
\begin{equation}\label{eq:flow-of-B*}
e^{B^*}\begin{pmatrix}
    w_0 \\ \dot{w}_0
\end{pmatrix}=
 \begin{pmatrix}
    w_0+\dot{w}_0 \\ \dot{w}_0
\end{pmatrix},
\end{equation}
In this way, we prove the instantaneous transition
\begin{equation} \label{exp(aB*)_target}
\begin{pmatrix}
    w_0 \\ \dot{w}_0
\end{pmatrix}
\quad \xrightarrow[]{\quad e^{B^*} \quad } \quad
 \begin{pmatrix}
    w_0 +   \dot{w}_{0} \\ \dot{w}_0
\end{pmatrix}.
\end{equation}
that change the profile while maintaining the velocity. This strategy allows to control the profile (see Section \ref{sec:vel->prof}), once the control on the velocity is established. 

\medskip

In this article, the main difficulties concerning this programme are:
\begin{itemize}[topsep=0cm, parsep=0cm, itemsep=0cm]
    \item the control system \eqref{wave_vect} is infinite-dimensional thus we must be vigilant about the norms for which convergences (such as \eqref{eq:commutator2}) hold, and whether they can be concatenated,
    \item the flows of $A_V$ can be used only forward in time, since it is a drift, and on time intervals as small as possible, to keep track of the minimum time.
\end{itemize}
More in general, the problem of understanding (the flow of) which commutators can be approximated by the dynamics of a control system with drift is a complex question, at the root of the nonlinear control theory of ODEs and PDEs.

\subsubsection{Propagation from well-prepared states and minimum time}

The proofs of \Cref{Prop:STAC} and \Cref{Prop:Tmin} are combinations of 2 elementary blocks:
\begin{itemize}[topsep=0cm, parsep=0cm, itemsep=0cm]
\item the small-time controllability of the velocity 
(\Cref{Prop:STAC_vel}) i.e. instantaneous transitions
$$\begin{pmatrix}
    w_0 \\ \dot{w}_0
\end{pmatrix}
\quad 
\xrightarrow[]{Z(w_0) \subset Z(\dot{w}_f-\dot{w}_0)} 
\quad
 \begin{pmatrix}
    w_0 \\ \dot{w}_f
\end{pmatrix}$$
\item the small time transfer of the velocity into the profile i.e. the instantaneous transition \eqref{exp(aB*)_target}.
\end{itemize}
For instance, we prove the first statement of \Cref{Prop:STAC} via the following 3 transitions
$$\begin{pmatrix}
    w_0 \\ \dot{w}_0
\end{pmatrix}
\quad \xrightarrow[]{\quad |Z(w_0)|=0 \quad } \quad
\begin{pmatrix}
    w_0 \\ w_f-w_0
\end{pmatrix}
\quad \xrightarrow[]{\quad e^{B^*} \quad } \quad
\begin{pmatrix}
    w_f \\ w_f-w_0
\end{pmatrix}
\quad \xrightarrow[]{\quad |Z(w_f)|=0 \quad} \quad
\begin{pmatrix}
    w_f \\ \dot{w_f}
\end{pmatrix}$$
The first transition corresponds to small-time control of the velocity because $|Z(w_0)|=0$ and so does the third one, because one may assume $|Z(w_f)|=0$. The second transition is a small-time transfer of the velocity into the profile: $w_f=w_0+(w_f-w_0)$.

These instantaneous transitions do not modify the zeroes set $Z(W)$. By \Cref{Prop:propagation}, the modification of the zeroes set requires time. The situation is perfectly clear in the specific case of a (non vanishing) initial condition $W_0=(0,\phi)$ with $\phi\geq 0$. Indeed, the free evolution of the wave equation produces in finite time a strictly positive profile, thus satisfying $|Z(W)|=0$, to which \Cref{Prop:STAC} can be applied. For instance, in dimension $d=1$, the D'Alembert formula
$$w(t,x)=\frac{1}{2}\int_{x-t}^{x+t}\phi(y)dy$$
proves that for every $t>r(W_0)=r(\phi)$ then $w(t,\cdot)>0$. 
A similar argument captures the minimal time in dimension $d=2$. In dimension $d \geq 3$, we only prove the positivity of the profile in sufficiently large time.
 
Finally, to prove the second point of \Cref{Prop:Tmin} or \Cref{Prop:large_time}, with a general initial condition $W_0=(w_0,\dot{w}_0)$, it suffices to prove that it can instantaneously reach a target $(0,\phi)$ with $\phi \geq 0$ and $Z(\phi)=Z(w_0)$. This is achieved with the following transitions
$$
 \begin{pmatrix}
    w_0 \\ \dot{w}_0
\end{pmatrix}
\quad \xrightarrow[]{\quad} \quad
\begin{pmatrix}
    w_0 \\ - \phi - w_0
\end{pmatrix}
\quad \xrightarrow[]{e^{B^*}} \quad
\begin{pmatrix}
    -\phi \\ -\phi - w_0
\end{pmatrix} 
\quad \xrightarrow[]{\quad} \quad 
\begin{pmatrix}
    -\phi \\ \phi
\end{pmatrix}
\quad \xrightarrow[]{e^{B^*}} \quad 
\begin{pmatrix}
    0 \\ \phi
\end{pmatrix},
$$
the 1st and 3rd ones are justified by 
$Z(w_0) \subset Z(\dot{w}_0+\phi+w_0)$ and 
$Z(\phi) \subset Z(2\phi+w_0)$.

\subsection{Bibliographical comments}

\subsubsection{Linear control of wave equations}

Intuitively, one can draw a parallel between bilinear control problems of the form $\partial_t^2 w - \Delta w + u w = 0$ and linear control problems of the form $\partial_t^2 w - \Delta w = \theta u$, where $\theta$ is a function that determines the control region. There is a vast amount of literature on the linear control of wave equations, and we restrict ourselves to citing only a few references. When $\theta$ is continuous and time-independent, the control region $\{\theta \neq 0\}$ is open, and the celebrated condition of Bardos--Lebeau--Rauch is known to be necessary and sufficient for controllability (see \cite{BLR, BurqGerard97}). Several extensions have been studied, such as the case of a subelliptic operator in \cite{Letrouit2022,Letrouit2023}, or the case of a rough boundary in \cite{BurqLeRousseauDehman1,BurqLeRousseauDehman2}.
Earlier results based on the multiplier method also provide controllability results in certain geometries. Some authors have addressed the case of a time-dependant open control region (see, for instance, \cite{LeRousseau17}). 
    
Recent works (see \cite{Burq2020, Rouveyrol2024}) have considered the related problem of stabilization with a damping function that is not continuous. In that setting, a different geometric condition arises; however, the damping is assumed to act in a region that is a union of polyhedra and, in particular, that has nonempty interior.
    
Less is known about the controllability of the wave equation when the control region is merely measurable. The dual question of observability has been studied in \cite{Humbert2019}, where it is shown that observability can be obtained under a geometric condition involving the average time spent by Riemannian geodesics in the observation domain. See also the very recent article \cite{niu2025symmetrywaveequationtorus} which studies the controllability of the 1D wave equation with a measurable space-time control region. In contrast, controllability of the heat equation with a rough control region has been studied by several authors; see, for instance, \cite{BurqMoyano2022} and the references therein. Finally, control problems for nonlinear wave equations has been studied in various situations; see, for instance, \cite{Coron2006, DehmanLebeau, Fattorini, Lasiecka05, Lasiecka1991, Perrin2026, TonBui, Zhou}.

\subsubsection{Exact bilinear control}

Let $s \in \mathbb{N}^*$. Since the potentials $V_j$ (see \eqref{eq:cos-sin_bis}) define bounded operators on $H^{s} \times H^{s-1}$, then System \eqref{wave_vect} is not exactly controllable in $H^{s} \times H^{s-1}(\T^d)$ with controls in $L^{1}_{loc}(\R^+,\R)$ (see \cite{BMS} for controls in $L^{1^+}_{loc}$ and \cite{chambrion-laurent} for controls in $L^{1}_{loc}$). 

The local exact controllability around $(w,\dot{w})=(1,0)$ is proved in a different functional framework in \cite{beauchard-laurent, beauchard-wave} when $d=1$, the control $u$ is scalar and the associated potential is appropriate. With such small controls, the minimal time is $T=2\pi$ (in this article no bounds are imposed on controls thus the minimal time can be smaller).

The proof uses the linear test and the solution of a trigonometric moment problem whose frequencies are the roots of the eigenvalues of the Laplacian. In the multi-dimensional case, this strategy fails because of the repartition of these eigenvalues (see 
\cite{haraux-jaffard}).

\subsubsection{Approximate bilinear control}

For $d=1$ and only one control in front of $V_0$, the system is $H^1\times L^2$-approximately controllable  in time $\geq 2\pi$ from initial states whose Fourier modes are all active \cite{BMS}.
The proof, relying on Fourier series, is limited to diagonal systems. Of course, this result is valid for our system \eqref{wave_vect} but the time $2\pi$ is not sharp (in fact, the minimal time depends on the initial condition)

Some large-time approximate controllability properties (towards particular final data) were also previously established for 1D non-linear damped wave equations with bilinear multiplicative control in \cite{khapalov-wave1} (see also \cite{khapalov-wave2}), via controlling a finite number of modes individually and subsequently in time (we notice that in those articles the control depends on both time and space variables).

In \cite{pozzoli} by the third author, small-time approximate controllability from initial states supported on a finite number of modes was established (hence, a sort of complementary result w.r.t. \cite{BMS}). There, the proof technique was also inspired by a Lie bracket approach, in particular exploiting commutators of the form \eqref{eq:commutator1}. Here, we develop further such Lie bracket approach and prove more general results.

\subsubsection{Agrachev-Sarychev method}
A part of our proof consists in showing that the flows along any potential $\mathcal{V}(x)\in L^\rho(\T^d)$ can be approximated (in small times) by the control system \eqref{eq:wave_bis} (more precisely, for any initial state $W_0\in H^1\times L^2(\T^d,\R)$, the orbit $\{e^{s\mathcal{V}B}W_0\}_{s\in \R}$ at $W_0$ along $\
\mathcal{V}B$, is small-time approximately reachable from $W_0$ for its vectorised version \eqref{wave_vect}). The proof of such property is based on a low-modes/degenerate forcing argument (generated by the commutators of the finite family of control potentials $V_0B,\dots,V_{2d}B
$ with the drift Laplacian $A+VB$) \emph{à la Agrachev-Sarychev} \cite{agrachev-sarychev,agrachev2}, firstly developed for the linear control of Navier-Stokes systems. As the space-dependent part of the control potential in \eqref{eq:wave_bis} is supported only on a finite number of low Fourier modes, it is localized in the frequency space, and such control is often called degenerate in control theory of PDEs \cite{rissel-nersesyan}. Such strategy was recently adapted to the bilinear control setting in \cite{duca-nersesyan} for controlling Schrödinger equations, and we refer also to \cite{coron-xiang-zhang,beauchard-pozzoli2} for additional recent applications of such strategies.

\subsection{Structure of the article}

In \Cref{sec:WP}, we recall classical well-posedness results concerning  equation \eqref{wave_vect} together with preliminary results. In \Cref{sec:CV}, we prove an abstract criterium for strong convergence that will be used several times in this article. In \Cref{sec:STAC_vel}, we prove the small-time controllability of the velocity, i.e. \Cref{Prop:STAC_vel}. 
In \Cref{sec:vel->prof}, we prove the small-time transfer of the velocity into the profile i.e. transitions \eqref{exp(aB*)_target}. In \Cref{sec:STAC}, we prove the small-time approximate controllability results, i.e. \Cref{Prop:STAC}. In \Cref{sec:Tmin}, we prove the approximate controllability results in optimal time, i.e. \Cref{Prop:Tmin}. In \Cref{sec:Large_time}, we prove the controllability in large enough time, i.e. \Cref{Prop:large_time}. In \Cref{sec:Discussion}, we discuss the assumptions and generalisations of these results.

\section{Well-posedness and preliminaries}
\label{sec:WP}

\subsection{Explicit group of isometries}

In this section, we recall classical properties of the group associated with the wave equation.

\begin{definition}
We define
\begin{itemize}[topsep=0cm, parsep=0cm, itemsep=0cm]
%
%
\item the norm $\|\cdot\|$ on $H^1 \times L^2(\T^d)$: for $W=(w_1,w_2) \in H^1 \times L^2(\T^d)$,
\begin{equation} \label{normalpha}
\|W\| = 
\left( \int_{\T^d} 
\left( 
|\nabla w_1 |^2  + 
|w_1|^2 + 
|w_2|^2  \right)  \right)^{\frac 12},
\end{equation}

\item the associated norm $\|\cdot\|$ for bounded operators on $H^1 \times L^2$:
for $L \in \mathcal{L}(H^1 \times L^2)$,
\begin{equation} \label{opnormalpha}
    \| L\|:=\sup\{ \|L(W)\| ; W \in H^1 \times L^2(\T^d), \|W\|=1 \},
\end{equation}
\item the unbounded operator 
\begin{equation} \label{def:Aalpha}
D(A)=H^2 \times H^1(\T^d), \qquad A = \begin{pmatrix}
    0 & 1 \\ \Delta-1 & 0 
\end{pmatrix}.
\end{equation}
\end{itemize}
\end{definition}

\begin{definition}
The Fourier coefficients of a function $f \in L^1(\T^d)$ are defined, for every $k \in \mathbb{Z}^d$ by
$$
c_k(f)=\frac{1}{|\T^d|} \int_{\T^d} f(x) e^{-i k \cdot x} dx. $$
\end{definition}

\begin{proposition} \label{Lem:A}
The unbounded operator $A$ generates a group of isometries of 
$\left( H^1 \times L^2(\T^d) , \|.\| \right)$: for $W_0=(w_0,\dot{w}_0) \in H^1 \times L^2(\T^d)$ and $t \in \R$ then $e^{tA}W_0 = (w(t),\dot{w}(t))$ where
\begin{equation} \label{SG:w1w2}
\begin{aligned}
w(t,x) & =\sum_{n\in\mathbb{Z}^d} \left( c_n(w_0) \cos(\langle n \rangle t) + c_n(\dot{w}_0) \frac{\sin(\langle n \rangle t)}{\langle n \rangle}  \right) e^{inx},
\\
\dot{w}(t,x) & =\sum_{n\in\mathbb{Z}^d} \Big( - \langle n \rangle c_n(w_0) \sin(\langle n \rangle t) + c_n(\dot{w}_0) \cos(\langle n \rangle t) \Big) e^{inx},
\\
\langle n \rangle & = \sqrt{1+|n|^2} .
\end{aligned}
\end{equation}
Moreover, the map $W:t \in \R \mapsto e^{tA} W_0$ belongs to $C^0(\R,H^1\times L^2(\T^d)) \cap C^1(\R,L^2\times H^{-1}(\T^d))$, satisfies $W(0)=W_0$ and the following equality in $L^2 \times H^{-1}(\T^d)$ for every $t \in \R$ 
$$ \frac{d W}{dt} (t) = A W(t). $$
\end{proposition}

\subsection{Generated groups}

In this section, we recall the construction of groups of bounded operators on $H^1 \times L^2(\T^d)$ via the Duhamel formula. This formula will be used to prove estimates in this article.

\begin{proposition} \label{Lem:A+L}
Let $L$ be a bounded operator on $H^1 \times L^2(\T^d)$. 
For every $W_0 \in H^1 \times L^2(\T^d)$, there exists a unique  
$W \in C^0(\R,H^1 \times L^2(\T^d))$ such that, for every $t \in \R$,
$$W(t)=e^{tA}W_0 + \int_0^t e^{(t-s)A} L W(s) ds.$$
Moreover $W \in  C^1(\R,L^2 \times H^{-1}(\T^d))$ and for every $t \in \R$,
$$\| W(t) \|_{H^1 \times L^2} \leq e^{\|L\|t} \|W_0\|_{H^1 \times L^2}
\qquad \text{ and } \qquad
\frac{dW}{dt}(t)=(A+L)W(t),$$
where the last equality holds in $L^2 \times H^{-1}(\T^d)$. 
\\
Therefore, the notation $\exp(t(A+L))W_0 := W(t)$ defines a group
$\big(\exp(t(A+L))\big)_{t\in\R}$ of bounded operators on $H^1 \times L^2(\T^d)$ such that
$$\| \exp(t(A+L)) \| \leq e^{\|L\|t}.$$
\end{proposition}

\begin{proof}
For the existence and uniqueness of $W$, we apply the fixed point theorem on\newline$C^0([0,T],H^1 \times L^2(\T^d))$ with $T \|L\|<1$. Since $T$ does not depend on $W_0$, the solution with the maximal definition interval is defined on $\R$. The $C^1$-regularity and the differential equation result from the ones of $t \mapsto e^{tA}$ (see \Cref{Lem:A}). For the bound, we apply Gronwall Lemma to the Duhamel formula. 
\end{proof}

\subsection{Operators associated with matrices}

The goal of this section is to prove the following result, where  $\rho$ is defined by \eqref{def:q_rho} and
\begin{equation} \label{def:rho_tild}
\widetilde{\rho} = \left\lbrace \begin{array}{ll}
2 & \text{ if } d \in \{1,2,3\}, \\
2^+ & \text{ if } d=4, \\
\frac{d}{2} & \text{ if } d \geq 5.
\end{array}\right.
\end{equation}

\begin{proposition} \label{Lem:M}
\begin{enumerate}[topsep=0cm, parsep=0cm, itemsep=0cm]
\item There exists $C>0$ such that, for every
$(m_1,m_2,m_3) \in  W^{1,\infty} \times L^{\rho} \times L^{\infty}(\mathbb{T}^d)$, the operator
\begin{equation} \label{def:M}
M:=\begin{pmatrix}
 m_{1} & 0 \\
 m_{2} & m_{3}
 \end{pmatrix}
\end{equation}
is bounded on $H^1 \times L^2(\T^d)$ and 
$\|M\| \leq C \big( \|m_1\|_{W^{1,\infty}}+\|m_{2}\|_{L^{\rho}}+\|m_{3}\|_{L^{\infty}} \big)$.
\item If $\phi \in L^{\rho}(\T^d)$ and $\nabla \phi \in L^{\widetilde{\rho}}(\T^d)$ then the operator
$\exp(\phi B)$ is bounded on $H^2 \times H^1(\T^d)$ (see \eqref{exp(phiB)}).
\end{enumerate}
\end{proposition} 

In particular, Propositions \ref{Lem:A+L} and \ref{Lem:M} prove the well-posedness stated in \Cref{Prop:WP} (for a constant control $u \in \R^{2d+1}$, take $L=m_2 B$ where $m_2=V+\sum_{j=0}^{2d} u_j V_j$ belongs to $L^{\rho}(\T^d)$).
The proof of \Cref{Lem:M} is an immediate consequence of the following statement.

\begin{lemma} \label{Lem:Sob}
Let $\rho, \widetilde{\rho}$ be as in \eqref{def:q_rho} and \eqref{def:rho_tild}.
\begin{enumerate}[topsep=0cm, parsep=0cm, itemsep=0cm]
\item If $\phi \in L^{\rho}(\T^d)$ then the operator
$w \in H^1(\T^d) \mapsto \phi w \in L^2(\T^d)$ is bounded.
\item If $\phi \in L^{\widetilde{\rho}}(\T^d)$ then the operator
$w \in H^2(\T^d) \mapsto \phi w \in L^2(\T^d)$ is bounded.
\item If $\phi \in L^{\rho}(\T^d)$ and $\nabla \phi \in L^{\widetilde{\rho}}(\T^d)$ then the operator
$w \in H^2(\T^d) \mapsto \phi w \in H^1(\T^d)$ is bounded.
\end{enumerate}
\end{lemma}

\begin{proof}
We first remark that 
\begin{equation} \label{step0}
q \in [2,\infty],\, w \in L^q(\T^d),\, \phi \in L^{\frac{2q}{q-2}}(\T^d)
\quad \Rightarrow \quad
\phi w \in L^2(\T^d)
\end{equation}
where, by convention $\frac{2q}{q-2}=2$ when $q=\infty$ and 
$\frac{2q}{q-2}=\infty$ when $q=2$. Indeed, when $q \in (2,\infty)$, by Hölder inequality,
$$\int_{\T^d} |\phi w|^2 \leq 
\left( \int_{\T^d} |\phi|^{2r} \right)^{\frac{1}{r}}
\left( \int_{\T^d} |w|^{2 \frac{q}{2}} \right)^{\frac{2}{q}}
= \|\phi\|_{L^{2r}}^2 \|w\|_{L^q}^2$$
where $r$ and $q/2$ are conjugate indices, i.e. $\frac{1}{r}+\frac{2}{q}=1$,
which implies $2r=\frac{2q}{q-2}$.
\\
\noindent \emph{1.} We have the Sobolev embedding $H^1(\T^d) \subset L^q(\T^d)$ where
\begin{itemize}[topsep=0cm, parsep=0cm, itemsep=0cm]
    \item if $d=1$, then $q=\infty$ and thus $\frac{2q}{q-2}=2$,
    \item if $d=2$, then $q$ can take any value in $[2,\infty)$ thus $\frac{2q}{q-2}$ can take any value $>2$, 
    \item if $d\geq 3$, then $q=\frac{2d}{d-2}$ thus $\frac{2q}{q-2}=d$.
\end{itemize}

\noindent \emph{2.} We have the Sobolev embedding $H^2(\T^d) \subset L^{\widetilde{q}}(\T^d)$ where
\begin{itemize}[topsep=0cm, parsep=0cm, itemsep=0cm]
    \item if $d \in \{1,2,3\}$, then $\widetilde{q}=\infty$ and thus $\frac{2\widetilde{q}}{\widetilde{q}-2}=2$,
    \item if $d=4$, then $\widetilde{q}$ can take any value in $[2,\infty)$ thus $\frac{2\widetilde{q}}{\widetilde{q}-2}$ can take any value $>2$, 
    \item if $d\geq 5$, then $\widetilde{q}=\frac{2d}{d-4}$ thus $\frac{2\widetilde{q}}{\widetilde{q}-2}=\frac{d}{2}$.
\end{itemize}

\noindent \emph{3.} Using the previous 2 statements, we obtain
$\| (\nabla \phi) w \|_{L^2} \leq \| \nabla \phi \|_{L^{\widetilde{\rho}}} \|w\|_{H^2}$ and
$\| \phi (\nabla w) \|_{L^2} \leq \| \phi \|_{L^{\rho}} \|w\|_{H^2}$, which gives the conclusion.
\end{proof}

\subsection{STAR operators}

We introduce the concept of $H^1\times L^2$-STAR operators to describe small-time approximately reachable targets.

\begin{definition}[$H^1\times L^2$-STAR operator] \label{def:STARop}
A bounded operator $L$ on $H^1 \times L^2(\T^d)$ is $H^1 \times L^2$-small-time approximately reachable if for every $W_0 \in H^1 \times L^2(\T^d)$ and $\epsilon>0$ there exist $T \in [0,\epsilon]$ and $u \in PWC([0,T],\R^{2d+1})$ such that $\| W(T;u,W_0) - LW_0 \|_{H^1 \times L^2}<\epsilon$.
\end{definition}

For instance, the small-time transition \eqref{exp(phiB)_target} corresponds to the $H^1 \times L^2$-STAR property for the operator 
\begin{equation} \label{exp(phiB)}
e^{\phi B}=\begin{pmatrix}
    1 & 0 \\ \phi & 1 
\end{pmatrix}. 
\end{equation}
This is not the case for the small-time transition \eqref{exp(aB*)_target} because the involved operator is not bounded on $H^1 \times L^2(\T^d)$. We easily deduce from \Cref{def:STARop} the following properties (see \Cref{Appendix:STAR} for a proof).

\begin{lemma} \label{Lem:STAR}
The composition and strong limit of $H^1 \times L^2$-STAR operators are $H^1 \times L^2$-STAR operators.
\end{lemma}

This lemma is the reason why we are interested in strong convergence criteria in the following section. We recall the following usual definition.

\begin{definition}
Let $(L_n)_{n\in\N}$ and L be bounded operators on $H^1 \times L^2(\T^d)$.
We say that $(L_n)_{n\in\N}$ strongly converges towards $L$ if, for every $W_0 \in H^1 \times L^2(\T^d)$, $\|(L_n-L) W_0 \| \to 0$ as $n \to \infty$.
\end{definition}

\section{A criterium for strong convergence}
\label{sec:CV}

The goal of this section is to prove the following abstract criterium of strong convergence. This criterium shows that it suffices to study the strong convergence 
at the level of generators. It is a classical tool when dealing with skew-adjoint generators (by combining, e.g., \cite[Theorem VIII.21 \& Theorem VIII.25(a)]{rs1}). We show it here in a different setting. It will be applied several times in this article.

\begin{proposition} \label{Lem:strongCV}
Let $\rho$ be as in \eqref{def:q_rho} and $(m_1^{\tau},m_2^{\tau},m_3^{\tau})_{\tau \in [0,1]}$ be bounded a family of $W^{1,\infty} \times L^{\rho} \times L^{\infty}(\mathbb{T}^d)$ for which the associated operators $M_{\tau}$ (see \eqref{def:M}) strongly converge towards $M_0$ on $H^1 \times L^2(\T^d)$. We assume that $(e^{tM_0})_{t\in\R}$ is a group of bounded operators on $H^2 \times H^1(\T^d)$. Then the operators
$\exp(\tau A + M_{\tau})$ 
strongly converge towards 
$\exp(M_0)$ as $\tau \to 0$.
\end{proposition}

\begin{proof}
By \Cref{Lem:M}, there exists $\ell>0$ such that, for every $\tau \in [0,1]$, $\|M_{\tau}\| \leq \ell$. 
By \Cref{Lem:A+L}, the operators $\exp(\tau A+M_{\tau})$ are bounded on $H^1 \times L^2(\T^d)$ uniformly with respect to $\tau \in [0,1]$: $\| \exp(\tau A+M_{\tau}) \| \leq e^{\ell}$.
Thus it suffices to prove the convergence in $H^1 \times L^2(\T^d)$ of $\exp(\tau A +M_{\tau})W_0$ for $W_0$ in a dense subset of $H^1 \times L^2(\T^d)$. 

Let $W_0 \in H^2 \times H^1(\T^d,\R^2)$ and $\tau \in (0,1]$. For $t \in \R$, we define
$W(t)=\exp(t(A+\frac{M_{\tau}}{\tau}))W_0$ and
$Z(t)=\exp(t \frac{M_0}{\tau})W_0$.
By definition (see \Cref{Lem:A+L}), the following equality holds in $H^1 \times L^2(\T^d)$ for every $t \in \R$,
$$W(t)=e^{tA}W_0 + \int_0^t e^{(t-s)A} \frac{M_{\tau}}{\tau} W(s) ds.$$
Since $(e^{tM_0})_{t\in\R}$ is a group of bounded operators on $H^2 \times H^1(\T^d)$ then $Z \in C^0(\R,H^2 \times H^1(\T^d))$ and satisfies the following equality in $H^1 \times L^2(\T^d)$ for every $t\in\R$,
$$Z(t)=e^{tA} W_0 + \int_0^t e^{(t-s)A} \left( \frac{M_0}{\tau}-A \right)Z(s) ds. $$
Thus, for every $t \in \R$,
$$(W-Z)(t)= \int_0^t e^{(t-s)A} 
\left( 
\frac{M_{\tau}}{\tau} (W-Z)(s) + 
\left( \frac{M_{\tau}-M_0}{\tau}  + A \right) Z(s)
\right) ds$$
and, by Gronwall Lemma
$$\| (W-Z)(\tau) \| \leq e^{\|M_{\tau}\|}
\left(
\frac{1}{\tau} \int_0^{\tau} \|(M_{\tau}-M_0)Z(s) \| ds
+
\int_0^{\tau} \|AZ(s)\| ds
\right) .$$
The first term in the right hand side is bounded by
$$
\|(M_{\tau}-M_0)W_0 \|
+ \frac{1}{\tau} \!\int_0^{\tau}\! \|(M_{\tau}-M_0)(Z(s)-W_0) \| ds
 \leq \!  \|(M_{\tau}-M_0)W_0 \|
+ 2 \ell \underset{s \in [0,\tau]}{\sup}\| Z(s)-W_0 \|.
$$
The second term in the right hand side is bounded by
$$\tau \underset{s \in [0,\tau]}{\sup} \|Z(s)\|_{H^2 \times H^1} 
\leq C \tau \|W_0\|_{H^2 \times H^1}.$$
These 2 bounds prove that 
$\| \exp(\tau A+M_{\tau})W_0-\exp(M_0)W_0 \| \to 0$ as $\tau \to 0$.
\end{proof}

\section{Small-time approximate control of the velocity}
\label{sec:STAC_vel}

The goal of this section is to prove \Cref{Prop:STAC_vel}. 
In  \Cref{subsec:SCV}, we prove strong convergences which form the basis of a saturation argument presented in \Cref{subsec:Saturation}. Finally, a density criterium presented in \Cref{subsec:DenseCrit} allows to end the proof of \Cref{Prop:STAC_vel} in \Cref{subsec:Proof_Vel}.

\subsection{Strong convergences}
\label{subsec:SCV}

\begin{proposition} \label{Cor:strongCV}
Let $A_V:=A+VB$ and $\phi \in W^{1,\infty}(\T^d)$. We have the following strong convergences of operators on $H^1 \times L^2(\T^d)$ as $\tau \to 0^+$
\begin{itemize}[topsep=0cm, parsep=0cm, itemsep=0cm]
 \item $\exp\big(\tau (A_V + \frac{\phi}{\tau}B) \big) \longrightarrow \exp(\phi B)$,
 \item $\exp(-\frac{\phi}{\sqrt{\tau}}B) \exp(\tau A_V) \exp( \frac{\phi }{\sqrt{\tau}}B) \longrightarrow \exp(-\phi^2 B)$.
\end{itemize}
\end{proposition}

We refer to \Cref{appendix:brackets} for an intuition, relying on Lie brackets, of these results. The first convergence actually holds as soon as $\phi \in L^{\rho}(\T^d)$ but we will not use it. These results are also proved in \cite{pozzoli}, but we propose here another proof strategy.

\begin{proof}
\noindent $\bullet$ We apply \Cref{Lem:strongCV} with $M_{\tau}=(\tau V+\phi) B$. By \Cref{Lem:M}, we have $\| M_{\tau}-M_0\| \leq C \|\tau V \|_{L^{\rho}} \to 0$ as $\tau \to 0$, i.e. $M_{\tau}-M_0$ converges to zero in $\mathcal{L}(H^1 \times L^2)$, thus $M_{\tau}$ strongly converges towards $M_0$ on $H^1 \times L^2(\T^d)$.
Moreover, by \Cref{Lem:M}, $(e^{t\phi B})_{t\in\mathbb{R}}$ is a group of bounded operators on $H^2 \times H^1(\T^d)$. Finally, by \Cref{Lem:strongCV}, $\exp(\tau A+M_{\tau})$ strongly converges towards $\exp(M_0)$ on $H^1 \times L^2$ when $\tau \to 0$, which gives the conclusion.   

\medskip

\noindent $\bullet$ We first prove that, for every $\gamma \in W^{1,\infty}(\T^d)$ and $t\in\R$,
\begin{equation} \label{formule1}
\exp(-\gamma B) \exp(t A_V) \exp( \gamma B) = \exp\big(t(A_V+\mathcal{M})\big) 
\quad \text{ where } \quad \mathcal{M}=\begin{pmatrix}
\gamma & 0 \\ -\gamma^2 & -\gamma
\end{pmatrix}
\end{equation}
(note that the coefficients of $VB+\mathcal{M}$ satisfy the assumptions of \Cref{Lem:M}).
Thanks to \eqref{def:B}, we obtain $\exp(\pm \gamma B)=I\pm\gamma B$, thus,
for $W_0=(w_1^0,w_2^0) \in H^1 \times L^2(\T^d)$, we have
$$\exp(-\gamma B) \exp(t A_V) \exp( \gamma B) W_0 =
\begin{pmatrix}
    w(t) \\ z(t)-\gamma w(t)
\end{pmatrix}
\quad \text{ where } \quad
\begin{pmatrix}
    w(t) \\ z(t)
\end{pmatrix}
= e^{t A_V} \begin{pmatrix}
    w_1^0 \\ w_2^0+\gamma w_1^0
\end{pmatrix}.$$
We have 
$$\frac{d}{dt} 
\begin{pmatrix}
    w \\ z-\gamma w
\end{pmatrix}
=
\begin{pmatrix}
    z \\ (\Delta+V)w-\gamma z
\end{pmatrix}
=\begin{pmatrix}
   \gamma w + (z-\gamma w) \\ (\Delta+V-\gamma^2)w-\gamma (z-\gamma w)
\end{pmatrix}
= (A_V+\mathcal{M}) \begin{pmatrix}
    w \\ z-\gamma w
\end{pmatrix},$$
and $(w,z-\gamma w)(0)=(w_1^0,w_2^0)$, thus
$$\begin{pmatrix}
    w(t) \\ z(t)-\gamma w(t)
\end{pmatrix}
= \exp \big(t(A_V+\mathcal{M}) \big) W_0$$
which proves \eqref{formule1}. Now, by applying \eqref{formule1} with $(t,\gamma) \leftarrow (\tau,\phi/\sqrt{\tau})$, we obtain
$$\exp \left(-\frac{\phi}{\sqrt{\tau}}B \right) \exp(\tau A_V) \exp \left( \frac{\phi}{\sqrt{\tau}} \right) 
= \exp(\tau A + M_{\tau}) 
\quad \text{ where } \quad
M_{\tau}=\begin{pmatrix}
\sqrt{\tau} \phi &  0 \\  \tau V  - \phi^2 & \sqrt{\tau} \phi
\end{pmatrix}.$$
We apply \Cref{Lem:strongCV}.
By \Cref{Lem:M},
$\|M_{\tau}-M_0\| \leq C ( \sqrt{\tau}\|\phi\|_{W^{1,\infty}} + \tau \|V\|_{L^{\rho}} ) \to 0$ as $\tau \to 0$, i.e. $M_{\tau}-M_0$ converges to zero in $\mathcal{L}(H^1 \times L^2)$ thus $M_{\tau}$ strongly converges towards $M_0$ on  $H^1 \times L^2(\T^d)$. Since $\phi^2 \in W^{1,\infty}(\T^d)$, by \Cref{Lem:M}, $(e^{tM_0})_{t\in\R}$ is a group of bounded operators on $H^2 \times H^1$. Finally, by \Cref{Lem:strongCV}, $\exp(\tau A+M_{\tau})$ strongly converges towards $\exp(M_0)$ on $H^1 \times L^2$ when $\tau \to 0$, which gives the conclusion.   
\end{proof}

\subsection{Saturation argument}
\label{subsec:Saturation}

In this section, we formalize the small-time transitions \eqref{exp(phiB)_target}, via \Cref{lem_psi}.

\begin{proposition} \label{Lem:psi_init}
    For every $\phi \in \text{Span}\{V_0,\dots,V_{2d}\}$, the operator
    $e^{\phi B}$ is $H^1 \times L^2$-STAR.
\end{proposition}

\begin{proof}
Let $\alpha \in \R^{2d+1}$ and $\phi=\alpha_0 V_0 + \dots + \alpha_{2d} V_{2d}$. For every $\tau>0$, the operator 
$\exp\big(\tau (A_V + \frac{\phi}{\tau}B) \big)$ is exactly reachable in time $\tau$ with the constant control $u=\alpha/\tau$. By \Cref{Cor:strongCV}, these operators strongly converge towards $e^{\phi B}$ on $H^1 \times L^2$. By \Cref{Lem:STAR}, $e^{\phi B}$ is $H^1 \times L^2$-STAR.
\end{proof}

\begin{proposition} \label{Lem:psi_iter}
Let $\phi \in W^{1,\infty}(\T^d)$.
We assume that, for every $\lambda \in \R$, the operator $e^{\lambda \phi B}$ is $H^1 \times L^2$-STAR. Then the operator $e^{-\phi^2 B}$ is $H^1 \times L^2$-STAR.
\end{proposition} 

\begin{proof}
By \Cref{Lem:STAR}, for every $\tau>0$, the operator 
$\exp(-\frac{\phi}{\sqrt{\tau}}B) \exp(\tau A_V) \exp( \frac{\phi}{\sqrt{\tau}})$ is $H^1 \times L^2$-approximately reachable in time $\tau^+$. By \Cref{Cor:strongCV}, these operators strongly converge towards $\exp(-\phi^2 B)$ on $H^1 \times L^2(\T^d)$ as $\tau \to 0$. By \Cref{Lem:STAR}, $\exp(-\phi^2 B)$ is $H^1 \times L^2$-STAR.
\end{proof}

\begin{proposition} \label{Lem:H}
Let $(\mathcal{H}_{\ell})_{\ell\in\mathbb{N}}$ be the increasing sequence of subspaces of $W^{1,\infty}(\T^d)$ defined in the following way:
\begin{itemize}[topsep=0cm, parsep=0cm, itemsep=0cm]
    \item $\mathcal{H}_0=\text{Span}\{V_0,\dots,V_{2d}\}$,
    \item for every $\ell \in \mathbb{N}^*$, $\mathcal{H}_{\ell}$ is the largest vector subspace whose elements $\phi$ can be written as
    \begin{equation} \label{dec:phi}
    \phi=\phi_0 - \sum_{i=1}^n \phi_i^2 \quad  \text{ where }  \quad
    n \in\mathbb{N}, \phi_0,\dots,\phi_n \in\mathcal{H}_{\ell-1}.
    \end{equation}
\end{itemize}
   Then $\mathcal{H}_{\infty}:=\cup_{\ell \in\mathbb{N}} \mathcal{H}_{\ell}$ 
   contains any trigonometric polynomial. Moreover, for every $\phi \in \mathcal{H}_{\infty}$, the operator $e^{\phi B}$ is $H^1 \times L^2$-STAR.
\end{proposition}

\begin{proof}
\emph{Step 1: We prove by induction on $\ell \in \mathbb{N}^*$ that for every $n \in \mathbb{N}^d$ with $|n| \leq \ell$ then $\cos(n\cdot x),\sin(n\cdot x) \in \mathcal{H}_{\ell-1}$.} The initialization for $\ell=1$ results from the definition of $\mathcal{H}_0$ and \eqref{eq:cos-sin_bis}. Let $\ell \in \mathbb{N}^*$ for which the property holds.
Let $n \in \mathbb{N}^d$ be such that $|n|=\ell+1$. Then $n=k+e_j$ for some $j \in \{1,\dots,d\}$ and $k \in \mathbb{N}^d$ with $|k|=\ell$. By induction the functions $1, \cos(x_j), \sin(x_j), \cos(k\cdot x), \sin(k\cdot x)$  belong to $\mathcal{H}_{\ell-1}$. The following trigonometric formula prove that the functions $\cos(n\cdot x), \sin(n\cdot x)$ belong to $\mathcal{H}_{\ell+1}$,
$$\begin{aligned}
\pm \cos(n\cdot x) & =\pm \cos(k\cdot x + x_j) =
1- \frac{1}{2} \left(\cos(k\cdot x)\mp\cos(x_j) \right)^2
- \frac{1}{2} \left(\sin(k\cdot x)\pm\sin(x_j) \right)^2,
\\
\pm \sin(n\cdot x) & =\pm \sin(k\cdot x + x_j )=1
-\frac{1}{2} \left(\sin(k\cdot x) \mp \cos(x_j) \right)^2 
-\frac{1}{2} \left(\cos(k\cdot x) \mp \sin(x_j) \right)^2.
\end{aligned}$$
\noindent \emph{Step 2: We prove by induction on $\ell \in \mathbb{N}$ that for every $\phi \in \mathcal{H}_{\ell}$, the operator $e^{\phi B}$ is $H^1 \times L^2$-STAR.} The initialization for $\ell=0$ is given by \Cref{Lem:psi_init}. Let $\ell \in \mathbb{N}^*$ and $\phi \in \mathcal{H}_{\ell}$. Then $\phi=\phi_0 - \sum_{i=1}^n \phi_i^2$ where $n \in\mathbb{N}$ and $\phi_0,\dots,\phi_n\in\mathcal{H}_{\ell-1}$.  For every $\lambda \in \mathbb{R}$ and $i \in \{0,\dots,n\}$, we have $\lambda \phi_i \in \mathcal{H}_{\ell-1}$ (because $\mathcal{H}_{\ell-1}$ is a vector space)
thus, by induction hypothesis, the operator $e^{\lambda \phi_i B}$ is $H^1 \times L^2$-STAR, therefore, by \Cref{Lem:psi_iter}, the operator $e^{-\phi_i^2 B}$ is $H^1 \times L^2$-STAR. Finally 
$e^{\phi B}=e^{\phi_0 B} e^{-\phi_1^2 B} \dots e^{-\phi_n^2 B}$
is a composition of $H^1 \times L^2$-STAR operators, thus an $H^1 \times L^2$-STAR operator by \Cref{Lem:STAR}.
\end{proof}

\begin{proposition} \label{lem_psi}
Let $\rho$ be as in \eqref{def:q_rho}.
For every $\phi \in L^{\rho}(\T^d)$, the operator $e^{\phi B}$ is $H^1 \times L^2$-STAR. 
\end{proposition}

\begin{proof}
By Fejer's theorem, there exists a sequence $(\phi_n)_{n\in\mathbb{N}}$ of trigonometric polynomials such that $\|\phi-\phi_n \|_{L^{\rho}} \to 0$ as $n \to \infty$. By \Cref{Lem:H}, for every $n \in \mathbb{N}$, $\phi_n \in \mathcal{H}_{\infty}$ thus the operator $e^{\phi_n B}$ is $H^1 \times L^2$-STAR. By \Cref{Lem:M},
$e^{\phi_n B}-e^{\phi B} = (\phi_n-\phi)B$ converges towards $0$ in  $\mathcal{L}(H^1 \times L^2)$, thus $e^{\phi_n B}$ strongly converges towards $e^{\phi B}$ on $H^1 \times L^2(\T^d)$. By \Cref{Lem:STAR}, $e^{\phi B}$ is $H^1 \times L^2$-STAR.
\end{proof}

\subsection{Density criterium}
\label{subsec:DenseCrit}

\begin{proposition} \label{Lem:dense}
Let $w_0 \in H^1(\T^d)$ and $\rho$ be as in \eqref{def:q_rho}.
\begin{enumerate}[topsep=0cm, parsep=0cm, itemsep=0cm]
 \item The subspace $\mathcal{F}:=\{f \in L^2(\T^d);f=0 \text{ a.e. on } Z(w_0) \}$
 is closed in $L^2(\T^d)$.
 \item The subspace $\{\phi w_0 ; \phi \in L^{\rho}(\T^d) \}$ is dense in $(\mathcal{F},\|.\|_{L^2})$. 
 \item If $|Z(w_0)|=0$ then $\{\phi w_0 ; \phi \in L^{\rho}(\T^d) \}$ is dense in $L^2(\T^d)$.
\end{enumerate}  
\end{proposition}  

\begin{proof}
\emph{1.} If $(f_n)_{n \in \mathbb{N}} \subset \mathcal{F}$, $f \in L^2(\T^d)$ and $\|f_n-f\|_{L^2} \to 0$ then   
$$\int_{Z(w_0)} |f|^2 = \int_{Z(w_0)} |f-f_n|^2 \leq \|f-f_n\|_{L^2} \to 0.$$
\noindent \emph{2.} It is equivalent to prove that its orthogonal is $=\{0\}$. Let $f \in \mathcal{F}$ be such that, for every $\phi \in L^{\rho}(\T^d)$,
$\int_{\T^d} f w_0 \phi=0$. Then $fw_0=0$ in $L^{\rho'}(\T^d)$, thus $f=0$ a.e. on $\T^d \setminus Z(w_0)$. Since $f=0$ a.e. on $Z(w_0)$
 then $f=0$. 
 \\
\noindent \emph{3.} If $|Z(w_0)|=0$ then $\mathcal{F}=L^2(\T^d)$ thus the second statement gives the conclusion.
\end{proof}

\subsection{Proof of \Cref{Prop:STAC_vel}}
\label{subsec:Proof_Vel}

Let $W_0=(w_0,\dot{w}_0)\in H^1 \times  L^2(\T^d)$ and $\dot{w}_f \in L^2(\T^d)$. We assume that $\dot{w}_f-\dot{w}_0$ vanishes a.e. on $Z(w_0)$. Let $\epsilon>0$ and $W_f=(w_0,\dot{w}_f)$.
With the notations of \Cref{Lem:dense}, we have $\dot{w}_f-\dot{w}_{0}  \in \mathcal{F}$ thus there exists $\phi \in L^{\rho}(\T^d)$ such that $\|\dot{w}_f-\dot{w}_0-\phi w_0\|<\frac{\epsilon}{2}$. Let $W_1:=(w_0,\dot{w}_0+\phi w_0)$. Since $W_1=e^{\phi B}W_0$, \Cref{lem_psi} provides $T \in (0,\epsilon)$ and $u \in PWC([0,T],\R^{2d+1})$ such that $\|W(T;u,W_0)-W_1\|<\frac{\epsilon}{2}$. Finally 
$\|W(T;u,W_0)-W_f\| \leq \|W(T;u,W_0)-W_1\| + \| W_1-W_f \| < \frac{\epsilon}{2} + \|\dot{w}_f-\dot{w}_0-\phi w_0\|_{L^2} <  \epsilon$.

\section{Small-time transfer of the velocity into the profile} \label{sec:vel->prof}

The goal of this section is to formalize the small time transitions \eqref{exp(aB*)_target} via the following statement.

\begin{proposition} \label{Prop:STAR_fin}
For every $W_0 \in H^1 \times H^1(\T^d)$ and $a \in (0,\infty)$ the target $e^{aB^*}W_0$ is small-time $H^1\times L^2$-approximately reachable.
\end{proposition}

Here, $B^*$ is the transpose of the matrix $B$ (see \eqref{def:B}).
The assumption $W_0=(w_0,\dot{w}_0) \in H^1 \times H^1(\T^d)$ is necessary for the target
$e^{a^* B} W_0 =(w_0+a\dot{w}_0,\dot{w}_0)$ to belong to $H^1 \times L^2(\T^d)$. 

In \Cref{subsec:dil}, we start with a preliminary result. In \Cref{subsec:CV_H1H1}, we prove the convergence result at the core of \Cref{Prop:STAR_fin}, which is proved in \Cref{subsec:Proof_vel->prof}

\subsection{Preliminary: small-time dilation/contraction}\label{subsec:dil}

The goal of this section is to prove the following result concerning the matrices
\begin{equation} \label{def:F}
    F=\begin{pmatrix}
     -1 & 0 \\ 0 & 1
 \end{pmatrix},
 \qquad \qquad
 \exp(\delta F) =\begin{pmatrix}
     e^{-\delta} & 0 \\ 0 & e^{\delta}
 \end{pmatrix}.
\end{equation}

\begin{proposition} \label{Prop:exp(deltaF)}
    For every $\delta \in \R$, the operator $\exp(\delta F)$ is $H^1 \times L^2$-STAR.
\end{proposition}

For proving \Cref{Prop:exp(deltaF)}, the key point is the following strong convergence.

\begin{proposition} \label{Cor:strongCV_bis}
Let $A_V:=A+VB$ and $\delta \in \R$. We have the following strong convergence of operators on $H^1 \times L^2(\T^d)$ as $\tau \to 0$
\begin{equation} \label{L_tau}
\exp\big(-\frac{\delta}{\tau} B\big) \exp\big( \tau(A_V+\frac{\delta^2}{\tau^2}B) \big) \exp\big(\frac{\delta}{\tau} B\big) \longrightarrow \exp(-\delta F).
\end{equation}
\end{proposition} 

We refer to \Cref{appendix:brackets} for an intuition of this result relying on Lie brackets.

\begin{proof}
    We first prove that, for every $\gamma, t \in \mathbb{R}$,
\begin{equation} \label{formule2}
\exp(-\gamma B) \exp\big(t (A_V+\gamma^2 B)\big) \exp( \gamma B) = \exp\big(t(A_V-\gamma F)\big).
\end{equation}
Indeed, for $W_0=(w_1^0,w_2^0) \in H^1 \times L^2(\T^d)$, we have
$$e^{-\gamma B} \exp\big(t (A_V+\gamma^2 B) \big) e^{ \gamma B} W_0 =
\begin{pmatrix}
    w(t) \\ z(t)-\gamma w(t)
\end{pmatrix}
\text{ where } 
\begin{pmatrix}
    w(t) \\ z(t)
\end{pmatrix}
= e^{t(A_V +\gamma^2 B)} \begin{pmatrix}
    w_1^0 \\ w_2^0+\gamma w_1^0
\end{pmatrix}.$$
thus
$$\frac{d}{dt} 
\begin{pmatrix}
    w \\ z-\gamma w
\end{pmatrix}
\!=\!
\begin{pmatrix}
    z \\ (\Delta+V+\gamma^2)w-\gamma z
\end{pmatrix}
\!=\!\begin{pmatrix}
   \gamma w + (z-\gamma w) \\ (\Delta+V)w-\gamma (z-\gamma w)
\end{pmatrix}
\!=\! (A_V-\gamma F) \begin{pmatrix}
    w \\ z-\gamma w
\end{pmatrix},$$
which gives the conclusion. By applying \eqref{formule2} with $(t,\gamma) \leftarrow (\tau,\delta/\tau)$ we obtain
$$ \exp\big(-\frac{\delta}{\tau} B\big) \exp\big( \tau(A+\frac{\delta^2}{\tau^2}B) \big) \exp\big(\frac{\delta}{\tau} B\big)
= \exp(\tau A + M_{\tau})
\quad \text{ where } \quad
M_{\tau} = \tau V B - \delta F.
$$
We apply \Cref{Lem:strongCV}.
By \Cref{Lem:M},
$\|M_{\tau}-M_0\| \leq C \tau \|V\|_{L^{\rho}}$ i.e.
$M_{\tau}-M_0$ converges to zero in $\mathcal{L}(H^1 \times L^2)$ thus
$M_{\tau}$ strongly converges towards $M_0=-\delta F$ on $H^1 \times L^2(T^d)$. Clearly, $(e^{-t\delta F})_{t\in\mathbb{R}}$ is a group of bounded operators on $H^2 \times H^1(\T^d)$. By \Cref{Lem:strongCV},
$\exp(\tau A + M_{\tau})$ strongly converges towards $\exp(M_0)$ on $H^1 \times L^2(\T^d)$.
\end{proof}

\begin{proof}[Proof of \Cref{Prop:exp(deltaF)}]
Let $\delta \in \R$. By \Cref{Lem:STAR} and \Cref{Lem:psi_init}, for every $\tau>0$, the operator $L_{\tau}$ in the left hand side of \eqref{L_tau}
is $H^1 \times L^2$-approximately reachable in time $\tau^+$. By \Cref{Cor:strongCV_bis}, $L_{\tau}$ converges strongly towards $\exp(-\delta F)$ on $H^1 \times L^2(\T^d)$. By \Cref{Lem:STAR}, $\exp(-\delta F)$ is $H^1 \times L^2$-STAR.
\end{proof}

\subsection{A convergence with loss of regularity}
\label{subsec:CV_H1H1}

The following convergence is the key point to prove Proposition \ref{Prop:STAR_fin}.

\begin{proposition} \label{Prop:H1*H1}
Let 
$a \in (0,\infty)$ and 
$W_0 \in H^1 \times H^1(\T^d)$. Then
$$\left\| \exp(\log(\tau)F) \exp( a \tau^2 A_V) \exp(-\log(\tau)F)W_0
- e^{a B^*}W_0
\right\|_{H^1\times L^2} \underset{\tau \to 0}{\longrightarrow} 0.$$
\end{proposition}
We refer to \Cref{appendix:brackets} for an intuition of this result relying on Lie brackets.

In the proof of \Cref{Prop:H1*H1}, we will use the following lemma.

\begin{lemma} \label{Lem:TP}
Let 
$\dot{w}_0 \in H^1(\T^d)$. 
We define
$(w(t),\dot{w}(t))=\exp(t A_V) \begin{pmatrix}
0 \\ \dot{w}_0
\end{pmatrix}$.
Then
$$\left\| \frac{1}{t} w(t) - \dot{w}_0 \right\|_{H^1} \underset{t \to 0}{\longrightarrow} 0. $$
\end{lemma}

\begin{proof}
\emph{First case $V=0$.} We define $(v,\dot{v})(t)=\exp(t A) \begin{pmatrix}
0 \\ \dot{w}_0
\end{pmatrix}$.
We deduce from \eqref{SG:w1w2} that
$$ \left\| \frac{1}{t} v(t) - \dot{w}_0 \right\|_{H^1}^2
= \sum_{n \in \mathbb{Z}} \left| 
\langle n \rangle c_n(\dot{w}_0) \left( 
\frac{\sin(\langle n \rangle t)}{\langle n \rangle t} -1 \right)
\right|^2
\underset{t \to 0}{\longrightarrow} 0$$
by the dominated convergence theorem.

\medskip

\noindent \emph{Second case $V \neq 0$.} 
We define $W(t)=(w,\dot{w})(t))$ and $W_0=(0,\dot{w}_0)$.
The Duhamel formula 
$$W(t)=e^{tA}W_0+\int_0^t e^{(t-s)A} VBW(s) ds$$
and the isometry of $e^{(t-s)A}$ on $H^1 \times L^2(\T^d)$ imply
$$\left\| \frac{1}{t}w(t)-\dot{w}_0 \right\|_{H^1} 
\leq \left\| \frac{1}{t} v(t) - \dot{w}_0 \right\|_{H^1}
+ \frac{1}{t} \int_0^t \| V B W(s) \|_{H^1 \times L^2}\, ds.$$
Moreover, by definition of $B$ and \Cref{Lem:Sob},
$$\| V B W(s) \|_{H^1 \times L^2}
= \|V w(s) \|_{L^2} \leq C \|V\|_{L^{\rho}} \|w(s)\|_{H^1}$$
thus
$$\left\| \frac{1}{t}w(t)-\dot{w}_0 \right\|_{H^1} 
\leq \left\| \frac{1}{t} v(t) - \dot{w}_0 \right\|_{H^1}
+ C \|V\|_{L^{\rho}} \int_0^1 \|w(\theta t)\|_{H^1} d\theta.$$
The first term in the right hand side converges to $0$ as $t \to 0$ by the first case. The second term also does because $\|w(s)\|_{H^1} \to 0$.
\end{proof}

\begin{proof}[Proof of \Cref{Prop:H1*H1}]
We define $W(\tau)=\exp(\log(\tau)F) \exp( a \tau^2 A_V) \exp(-\log(\tau)F)W_0$.
We deduce from \eqref{def:F} that
$$W(\tau)
 =\begin{pmatrix}
\frac{1}{\tau} & 0 \\ 0 & \tau
\end{pmatrix}
\exp( a \tau^2 A_V) 
\begin{pmatrix}
\tau & 0 \\ 0 & \frac{1}{\tau}
\end{pmatrix}
\begin{pmatrix}
w_0 \\ \dot{w}_0
\end{pmatrix}.$$
Using linearity, the explicit expression of $e^{aB^*}$ and the triangle inequality, we obtain
$$\begin{aligned}
\left\| W(\tau)-e^{aB^*}W_0 \right\|
& \leq 
\left\| \begin{pmatrix}
1 & 0 \\ 0 & \tau^2
\end{pmatrix}
\exp( a \tau^2 A_V) 
\begin{pmatrix}
w_0 \\ 0
\end{pmatrix} - 
\begin{pmatrix}
w_0 \\ 0
\end{pmatrix}
\right\| 
\\ & +  
\left\| 
\begin{pmatrix}
\frac{1}{\tau^2} & 0 \\ 0 & 1
\end{pmatrix}
\exp( a \tau^2 A_V) 
\begin{pmatrix}
0 \\ \dot{w}_0
\end{pmatrix}
- 
\begin{pmatrix}
a \dot{w}_0 \\ \dot{w}_0
\end{pmatrix}
\right\|.
\end{aligned}$$
The first term in the right hand side converges to zero as $\tau \to 0$ because so does\newline$\| \exp(a\tau^2 A_V)(w_0,0) - (w_0,0)\|$ since $(w_0,0) \in H^1 \times L^2(\T^d)$. In the second term, the first component converges to $0$ in $H^1$ by \ \Cref{Lem:TP} and the second component converges to $0$ in $L^2$ because
$\|\exp(a\tau^2 A_V)(0,\dot{w}_0) - (0,\dot{w}_0)\| \to 0$ since $(0,\dot{w}_0) \in H^1 \times L^2(\T^d)$.
\end{proof}

\subsection{Proof of \Cref{Prop:STAR_fin}}
\label{subsec:Proof_vel->prof}

    Let $W_0 \in H^1 \times H^1(\T^d)$, $a \in (0,\infty)$ and $\epsilon>0$. By \Cref{Prop:H1*H1}, there exists $\tau>0$ such that $\| L_{\tau}W_0 - e^{a B^*}W_0\| < \frac{\epsilon}{2}$,
where $L_{\tau} = \exp(\log(\tau)F) \exp( a \tau^2 A_V) \exp(-\log(\tau)F)$. Moreover, one may assume $\tau$ small enough so that $a\tau^2<\epsilon$. By \Cref{Prop:exp(deltaF)} and \Cref{Lem:STAR}, the operator $L_{\tau}$ is $H^1 \times L^2$-approximately reachable in time $(a\tau^2)^+$. Thus there exists $T \in (0,\epsilon)$ and $u \in PWC([0,T],\R^{2d+1})$ such that 
$\|W(T;u,W_0)-L_{\tau}W_0\| < \frac{\epsilon}{2}$. By the triangular inequality, we conclude $\| W(T;u,W_0)-e^{aB^*}W_0\|<\epsilon$.

\section{Small-time approximate controllability}
\label{sec:STAC}

In this section, we prove \Cref{Prop:STAC}. 
In \Cref{subsec:STAC_1}, we prove the first statement of \Cref{Prop:STAC}.
In \Cref{subsec:W1}, we formalize an intermediary result that will be useful in the end of this article. Finally, in \Cref{subsec:STAC_2}, we prove the second statement of \Cref{Prop:STAC}.

\subsection{Proof of the first statement of \Cref{Prop:STAC}}    
\label{subsec:STAC_1}

Let $W_0=(w_0,\dot{w}_0) \in H^1 \times L^2(\T^d)$ with $|Z(w_0)|=0$.
We consider a target $W_f=(w_f,\dot{w}_f) \in H^1 \times L^2(\T^d)$ and $\epsilon>0$. One may assume that $|Z(w_f)|=0$ (otherwise modify $w_f$ and reduce $\epsilon$). We will use the following steps in our control strategy
$$W_0=\begin{pmatrix}
    w_0 \\ \dot{w}_0
\end{pmatrix}
\xrightarrow[T_1,u_1]{} 
W_1=\begin{pmatrix}
    w_0 \\ w_f-w_0
\end{pmatrix}
\xrightarrow[T_2,u_2]{e^{B^*}} 
W_2=\begin{pmatrix}
    w_f \\ w_f-w_0
\end{pmatrix}
\xrightarrow[T_3,u_3]{}
W_f=\begin{pmatrix}
    w_f \\ \dot{w_f}
\end{pmatrix}.$$

\noindent $\bullet$ Since $|Z(w_f)|=0$, \Cref{Prop:STAC_vel} provides $T_3 \in (0,\frac{\epsilon}{3})$ and $u_3 \in PWC([0,T_3],\R^{2d+1})$ such that
$\| W(T_3;u_3,W_2)-W_f\|<\frac{\epsilon}{2}$. By continuity of $W(T_3;u_3,\cdot)$ on $H^1 \times L^2(\T^d)$, there exists $\delta_3>0$ such that, for every $\widetilde{W}_0 \in H^1 \times L^2(\T^d)$ with $\|\widetilde{W}_0-W_2\|<\delta_3$ then $\| W(T_3;u_3,\widetilde{W}_0)-W_f\|<\epsilon$.

\medskip

\noindent $\bullet$ Since $W_1 \in H^1 \times H^1(\T^d)$ and $e^{B^*}W_1=W_2$, by \Cref{Prop:STAR_fin}, there exists $T_2 \in (0,\frac{\epsilon}{3})$ and $u_2 \in PWC([0,T_2],\R^{2d+1})$ such that
$\| W(T_2;u_2,W_1)- W_2\|<\frac{\delta_3}{2}$.
By continuity of $W(T_2;u_2,\cdot)$ on $H^1 \times L^2(\T^d)$, there exists $\delta_2>0$ such that, for every $\widetilde{W}_0 \in H^1 \times L^2(\T^d)$ with $\|W_1-\widetilde{W}_0\|<\delta_2$ then $\| W(T_2;u_2,\widetilde{W}_0)-W_2\|<\delta_3$.

\medskip

\noindent $\bullet$ Since $|Z(w_0)|=0$, \Cref{Prop:STAC_vel} provides $T_1 \in (0,\frac{\epsilon}{3})$ and $u_1 \in PWC([0,T_1],\R^{2d+1})$ such that
$\| W(T_1;u_1,W_0)-W_1\|<\delta_2$.

\medskip

Finally, with $T=T_1+T_2+T_3 \in (0,\epsilon)$ and $u:[0,T] \to \R^{2d+1}$ the concatenation of $u_1$, $u_2$ and $u_3$, we obtain
$\|W(T;u,W_0)-W_f\|<\epsilon$.

\subsection{Reaching $W_1$ such that $Z(w_1)=Z(W_1)=Z(W_0)$}    
\label{subsec:W1}

\begin{proposition} \label{Lem:w0+aw0.}
Let $W_0=(w_0,\dot{w}_0) \in H^1 \times H^1(\T^d)$ be a non zero initial condition. 
\begin{enumerate}[topsep=0cm, parsep=0cm, itemsep=0cm]
\item  There exists $a \in (0,\infty)$ such that $|Z(w_0+a\dot{w}_0)|=|Z(W_0)|$.
\item There exists a target $W_1=(w_1,\dot{w}_1) \in H^1 \times L^2(\T^d)$ small-time $H^1 \times L^2$-approximately reachable from $W_0$ such that $Z(w_1)=Z(W_1)=Z(W_0)$.
\end{enumerate}
\end{proposition}

\begin{proof}
\emph{1.} For any $a \in \R$, $Z(W_0) \subset Z(w_0+a\dot{w}_0)$.
For $a \in \R$, $E_a :=Z(w_0+a\dot{w}_0) \setminus Z(W_0)$ defines a familly $(E_a)_{a\in\R}$ of measurable subsets of $\T^d$ such that for any $a \neq b$, $E_a \cap E_b = \emptyset$. For any finite subset $A$ of $\R$, we have
$$\sum_{a \in A} |E_{a}| = \left| \underset{a \in A}{\bigsqcup} E_a \right| \leq | \T^d|$$
thus the familly $(|E_a|)_{a\in\R}$ is summable. As a consequence, it has a countable support. In particular, there exists $a \in \R_+$ such that $|E_a|=0$.
\\
\emph{2.} Let $a \in (0,\infty)$ be given by the Statement 1. By \Cref{Prop:STAR_fin}, the target $W_1=e^{aB^*}W_0$ is small-time $H^1 \times L^2$-approximately reachable from $W_0$.
\end{proof}

\subsection{Proof of the second statement of \Cref{Prop:STAC}}    
\label{subsec:STAC_2}

Let $W_0=(w_0,\dot{w}_0) \in H^1 \times H^1(\T^d)$ with $|Z(W_0)|=0$. We consider a target $W_f=(w_f,\dot{w}_f) \in H^1 \times L^2(\T^d)$ and $\epsilon>0$. By \Cref{Lem:w0+aw0.}, there exists a target $W_1=(w_1,\dot{w}_1) \in H^1 \times L^2(\T^d)$ small-time $H^1 \times L^2$-approximately reachable from $W_0$, such that $Z(w_1)=Z(W_1)=Z(W_0)$. In particular, we have $|Z(w_1)|=0$. We use the following steps in our control strategy
$$W_0=\begin{pmatrix}
    w_0 \\ \dot{w}_0
\end{pmatrix}
\quad \xrightarrow[T_1,u_1]{} \quad
W_1=\begin{pmatrix}
    w_1 \\ \dot{w}_1
\end{pmatrix}
\quad \xrightarrow[T_2,u_2]{} \quad
W_f=\begin{pmatrix}
    w_f \\ \dot{w_f}
\end{pmatrix}.$$

\noindent $\bullet$ Since $|Z(w_1)|=0$, the first statement of \Cref{Prop:STAC} provides $T_2 \in (0,\frac{\epsilon}{2})$ and\newline$u_2 \in PWC([0,T_2],\R^{2d+1})$ such that
$\| W(T_2;u_2,W_1)-W_f\|<\frac{\epsilon}{2}$. By continuity of $W(T_2;u_2,\cdot)$ on $H^1 \times L^2(\T^d)$, there exists $\delta_2>0$ such that, for every $\widetilde{W}_0 \in H^1 \times L^2(\T^d)$ with $\|W_1-\widetilde{W}_0\|<\delta_2$ then $\| W(T_2;u_2,\widetilde{W}_0)-W_f\|<\epsilon$.

\medskip

\noindent $\bullet$ By definition of $W_1$, there exists $T_1 \in (0,\frac{\epsilon}{2})$ and $u_1 \in PWC([0,T_1],\R^{2d+1})$ such that\newline$\| W(T_1;u_1,W_0)-W_1\|<\delta_2$.

\medskip

Finally, with $T=T_1+T_2 \in (0,\epsilon)$ and $u:[0,T] \to \R^{2d+1}$ the concatenation of $u_1$ and $u_2$, we obtain
$\|W(T;u,W_0)-W_f\|<\epsilon$.

\section{Minimum time for approximate controllability}
\label{sec:Tmin}

In this section, we prove \Cref{Prop:Tmin}. 

In \Cref{subsec:profile>0}, we identify initial conditions $W_0$ from which one can reach in time $r(W_0)$ a state $W_1$ with strictly positive profile (thus to which the first statement of \Cref{Prop:STAC} applies).
These initial conditions have the form $W_0=(0,\dot{w}_0)$ with $\dot{w}_0 \geq 0$ (resp. $\leq 0$) a.e. on $\T^d$. This step uses explicit representation formula of the solution of the wave equation in dimension $d=1$ and $d=2$. The analogue result is false in dimension $d=3$.

In \Cref{subsec:Reach0phi}, we explain how to reach in arbitrarily small time such a state from any initial condition. This allows to write the proof of \Cref{Prop:Tmin} in \Cref{subsec:proofTmin}.

\subsection{Reaching positive profiles}
\label{subsec:profile>0}

We introduce
\begin{equation} \label{def:Atilde}
\widetilde{A}:=\begin{pmatrix}
    0 & 1 \\ \Delta & 0
\end{pmatrix}.
\end{equation}
Then $\widetilde{A}=A+\sum_{j=0}^{2d} \alpha_j V_j B$ where $\alpha_0=1$ and $\alpha_j=0$ for $j \geq 1$ thus the group $(e^{t \widetilde{A}})_{t\in\R}$ corresponds to trajectories of \eqref{wave_vect} associated with the constant control $u=e_0$ (where $e_0=(1,0,\dots,0)$ denotes the first element of the canonical basis of $\R^{2d+1}$) and $V=0$.

\begin{proposition} \label{Lem:signe}
We assume $d \in \{1,2\}$. Let $W_0=(0,\dot{w}_0) \in H^1 \times L^2(\T^d)$ be a non zero initial condition such that 
$\dot{w}_0 \geq 0$ a.e. on $\T^d$. Let $(w(t),\partial_t w(t)):=e^{t\widetilde{A}}W_0$ for $t \in \R$. Then, for every $t>r(W_0)$, we have $w(t,\cdot)>0$ a.e. on  $\T^d$.
\end{proposition}

\begin{proof}
\noindent \emph{Case $d=1$:} The following explicit expression, consequence of D'Alembert formula,
\begin{equation}\label{eq:DAlembert_dot_w0}
w(t,x)=\frac{1}{2}\int_{x-t}^{x+t} \dot{w}_0 
\end{equation}
proves that, for every $x \in\T$ and $t>r(\dot{w}_0)$ then $w(t,x)>0$. Note that \eqref{eq:DAlembert_dot_w0} defines an element of $C^0(\R; H^1(\T)) \cap C^1(\R; L^2(\T))$, by continuity of translations on $L^2(\T)$, and therefore it is indeed the solution of the wave equation with initial data $(0, \dot{w}_0)$.

\medskip

\noindent \emph{Case $d=2$:} For a sufficiently regular data $\dot{w}_0$, the following explicit expression, consequence of Poisson's formula (see, e.g., \cite{evans-book}, Section 4.2.1, equation (27) on p. 74, and Theorem 3 on p. 80) 
\begin{equation}\label{eq:proof:Lem:signe:1}
w(t,x)=\frac{1}{2\pi t^2} \int_{B(x,t)}
\frac{t^2\dot{w}_0(y)}{(t^2-|y-x|^2)^{1/2}} dy    
\end{equation}
proves that, if $x \in\T^2$ and $t>r(\dot{w}_0)$, then $w(t,x)>0$.

To complete the proof, we need to show that, at low regularity, i.e. $\dot{w}_0 \in L^2(\T^d)$, the solution $w$ (given by Proposition \ref{Lem:A}) coincides with the function at the right-hand side of \eqref{eq:proof:Lem:signe:1}, for all $t > 0$ and almost all $x \in \mathbb{T}^2$. We denote by $w$ both the solution on $\mathbb{T}^2$ and its periodic extension as a function on $\mathbb{R}^2$. For $t > 0$ and $y \in \mathbb{R}^2$, set 
\begin{equation*}
    K_t(y) = \mathbb{1}_{B(0,t)}(y) \frac{1}{\sqrt{t^2 - \vert y \vert^2}}.
\end{equation*}
Using polar coordinates, one finds
\begin{equation*}
    \Vert K_t \Vert_{L^1(\mathbb{R}^2)} 
    = \int_{B(0,t)} \frac{d y}{\sqrt{t^2 - \vert y \vert^2}} 
    = 2 \pi \int_0^t \frac{r}{\sqrt{t^2 - r^2}} d r 
    < + \infty.
\end{equation*}
Let $t > 0$ and $g \in L^2(\mathbb{T}^2)$, extended by periodicity. Then Young's inequality gives
\begin{align*}
    \Vert K_t \ast g \Vert_{L^2(\mathbb{T}^2)}^2 
    & \leq \int_{\mathbb{R}^2} \left\vert K_t \ast \left( g \mathbb{1}_{[0, 2\pi]^2 + B(0,t)} \right) \right\vert^2(x) d x \\
    & \leq \Vert K_t \Vert_{L^1(\mathbb{R}^2)}^2 \Vert g \Vert_{L^2([0, 2\pi]^2 + B(0,t))}^2 
    \leq C_t \Vert g \Vert_{L^2(\mathbb{T}^2)}^2,
\end{align*}
for some $C_t > 0$, where the sum of two sets is defined by $E_1 + E_2 = \{x+y, x \in E_1, y \in E_2\}$. Hence, the operator 
\begin{equation*}
    \dot{w}_0 \in L^2(\mathbb{T}^2) \longmapsto \frac{1}{2 \pi} K_t \ast \dot{w}_0 \in L^2(\mathbb{T}^2)
\end{equation*}
is well-defined and continuous. By density of smooth functions in $L^2(\mathbb{T}^2)$, this proves that Poisson's formula
indeed coincides punctually with the solution given by Proposition \ref{Lem:A}.
\end{proof}

\medskip

\begin{remark}
    \Cref{Lem:signe} does not hold for $d=3$. To justify this, we use Kirchhoff’s representation formula (see \cite{evans-book}, Section 4.2.1, equation (22) on p. 69, and Theorem 2 on p. 73) 
    \begin{equation} \label{Ondes_3d}
    w(t,x)=t \int_{\partial B (x,t)} \dot{w}_0(y) d\sigma_t (y),
    \end{equation}
    and a well-chosen initial condition 
    $\dot{w}_0 \in C^2(\T^3)$ satisfying the assumptions of \Cref{Lem:signe}.
    The formula \eqref{Ondes_3d} proves that, for any point $x \in \T^d$, there exists a time $t \in \R_+$ at which $w(t,x)>0$. But this time may fail to be $r(W_0)^+$ uniformly with respect to $x$. 
    
    We give an example to illustrate this phenomenon. 
    Let us consider
    $0<R'<R<\frac{\pi}{3}$, 
    $x_0 \in \T^d$ and 
    $K=\{ x \in \T^3 ; |x-x_0| \notin (R',R'+2R) \}$ 
    which is the complementary of a ring surrounding $x_0$.
    The initial velocity $\dot{w}_0$ is an element of $C^{\infty}(\T^3,\R^+)$ such that $\text{Supp}(\dot{w}_0)=K$.
    For $\epsilon>0$ small enough and $x \in B(x_0,\epsilon)$, the function
    $t \mapsto w(t,x)$ vanishes on $(R'+\epsilon,R'+2R-\epsilon)$, because the integrals in \eqref{Ondes_3d} are on spheres contained in the ring $\T^d \setminus K$. Thus 
    $B(x_0,\epsilon) \subset Z( w(t,.) )$ for any $t \in (R'+\epsilon,R'+2R-\epsilon)$, which is a neighborhood of $R$. 
    
\end{remark}

\subsection{Small-time reachability of $(0,\varphi)$}
\label{subsec:Reach0phi}

\begin{proposition} \label{Prop:reach0varphi}
Let $W_0=(w_0,\dot{w}_0) \in H^1 \times L^2(\T^d)$ and $\varphi \in H^1(\T^d)$ such that $Z(w_0)=Z(W_0)=Z(\varphi)$. The target $(0,\varphi)$ is small-time $H^1 \times L^2$-approximately reachable from $W_0$.
\end{proposition}

\begin{proof}
Let $\epsilon>0$. We use the following steps in our control strategy
$$\begin{aligned}
& W_0=\begin{pmatrix}
    w_0 \\ \dot{w}_0
\end{pmatrix}
 \quad \xrightarrow[T_1,r_1]{} \quad
W_1=\begin{pmatrix}
    w_0 \\ - \varphi - w_0
\end{pmatrix}
\quad \xrightarrow[T_2,u_2]{e^{B^*}} \quad
W_2=\begin{pmatrix}
    -\varphi \\ -\varphi - w_0
\end{pmatrix}
\\ & 
\quad \xrightarrow[T_3,u_3]{} \quad
W_3=\begin{pmatrix}
    -\varphi \\ \varphi
\end{pmatrix}
\quad \xrightarrow[T_4,u_4]{e^{B^*}} \quad
W_4=\begin{pmatrix}
    0 \\ \varphi
\end{pmatrix}.
\end{aligned}$$
$\bullet$ Since $W_3 \in H^1 \times H^1(\T^d)$ and $W_4=e^{B^*}W_3$, by \Cref{Prop:STAR_fin}, there exists $T_4 \in (0,\frac{\epsilon}{4})$ and $u_4 \in PWC((0,T_4),\R^{2d+1})$
such that 
$\| W(T_4;u_4,W_3)-W_4\|<\frac{\epsilon}{2}$. By continuity of $W(T_4;u_4,.)$ on $H^1 \times L^2(\T^d)$, there exists $\delta_4>0$ such that, for every $\widetilde{W}_0 \in H^1 \times L^2(\T^d)$ with $\|\widetilde{W}_0-W_3\|<\delta_4$ then $\| W(T_4;u_4,\widetilde{W}_0)-W_4\|<\epsilon$

\medskip

\noindent $\bullet$ Since $2\varphi+w_0=0$ a.e. on $Z(\varphi)$, by \Cref{Prop:STAC_vel}, there exists $T_3 \in (0,\frac{\epsilon}{4})$ and\newline$u_3 \in PWC((0,T_3),\R^{2d+1})$ such that 
$\| W(T_3;u_3,W_2)-W_3\|<\frac{\delta_4}{2}$. By continuity of $W(T_3;u_3,.)$ on $H^1 \times L^2(\T^d)$, there exists $\delta_3>0$ such that, for every $\widetilde{W}_0 \in H^1 \times L^2(\T^d)$ with $\|\widetilde{W}_0-W_2\|<\delta_3$ then $\| W(T_3;u_3,\widetilde{W}_0)-W_3\|<\delta_4$.

\medskip

\noindent $\bullet$ Since $W_1 \in H^1 \times H^1(\T^d)$ and $W_2=e^{B^*}W_1$, by \Cref{Prop:STAR_fin}, there exists $T_2 \in (0,\frac{\epsilon}{4})$ and $u_2 \in PWC((0,T_2),\R^{2d+1})$
such that 
$\| W(T_2;u_2,W_1)-W_2\|<\frac{\delta_3}{2}$. By continuity of $W(T_2;u_2,.)$ on $H^1 \times L^2(\T^d)$, there exists $\delta_2>0$ such that, for every $\widetilde{W}_0 \in H^1 \times L^2(\T^d)$ with $\|\widetilde{W}_0-W_1\|<\delta_2$ then $\| W(T_2;u_2,\widetilde{W}_0)-W_2\|<\delta_3$.

\medskip 

\noindent $\bullet$ Since $-\varphi-w_0-\dot{w}_0=0$ a.e. on $Z(w_0)$, by \Cref{Prop:STAC_vel}, there exists $T_1 \in (0,\frac{\epsilon}{4})$ and $u_1 \in PWC((0,T_1),\R^{2d+1})$ such that 
$\| W(T_1;u_1,W_0)-W_1\|<\delta_2$. 

\medskip

Finally, with $T=T_1+\dots+T_4 \in (0,\epsilon)$ and $u:[0,T] \to \R^{2d+1}$ the concatenation of $u_1$, \dots, $u_4$, we obtain $\|W(T;u,W_0)-W_4\|<\epsilon$.
\end{proof}

\subsection{Proof of \Cref{Prop:Tmin}}
\label{subsec:proofTmin}

We consider a target $W_f=(w_f,\dot{w}_f) \in H^1 \times L^2(\T^d)$ and $\epsilon>0$.

\bigskip

\noindent \emph{Proof of the first statement of \Cref{Prop:Tmin}:}
Let $W_0=(0,\dot{w}_0) \in H^1 \times L^2(\T^d)$ be a non zero initial condition such that $\dot{w}_0 \geq 0$ a.e. on  $\T^d$. We use the following steps in our control strategy
$$W_0=\begin{pmatrix}
    0 \\ \dot{w}_0
\end{pmatrix}
\quad \xrightarrow[u=e_0]{} \quad
W_1=\begin{pmatrix}
    w_1 \\ \dot{w}_1
\end{pmatrix}=e^{(r(W_0)+\frac{\epsilon}{2}) \widetilde{A}}W_0
\quad \xrightarrow[T_2,u_2]{} \quad
W_f=\begin{pmatrix}
    w_f \\ \dot{w_f}
\end{pmatrix}.$$
By \Cref{Lem:signe}, $|Z(w_1)|=0$ thus the first statement of \Cref{Prop:STAC} provides $T_2 \in (0,\frac{\epsilon}{2})$ and $u_2 \in PWC([0,T_2],\R^{2d+1})$ such that $\|W(T_2;u_2,W_1)-W_f\|<\epsilon$. Finally, with $T=r(W_0)+\frac{\epsilon}{2}+T_1 \in (r(W_0),r(W_0)+\epsilon)$ and $u:[0,T] \to \R^{2d+1}$ the concatenation of $e_0$ and $u_2$, we obtain $\|W(T;u,W_0)-W_f\|<\epsilon$.

\bigskip

\noindent \emph{Proof of the second statement of \Cref{Prop:Tmin}:}
Let $W_0=(w_0,\dot{w}_0) \in H^1 \times L^2(\T^d)$ be a non zero initial condition such that $\dot{w}_0=0$ a.e. on $Z(w_0)$. Then $w_0 \neq 0$. We use the following steps in our control strategy
$$\begin{aligned}
& W_0=\begin{pmatrix}
    w_0 \\ \dot{w}_0
\end{pmatrix}
\quad \xrightarrow[T_1,u_1]{} \quad
W_1=\begin{pmatrix}
    0 \\ -|w_0|
\end{pmatrix}
\quad \xrightarrow[T_2,u_2]{} \quad
W_f=\begin{pmatrix}
    w_f \\ \dot{w_f}
\end{pmatrix}.
\end{aligned}$$

\medskip

\noindent $\bullet$ Since $Z(W_1)=Z(W_0)$, by the first statement of \Cref{Prop:Tmin}, there exist 
$T_2 \in (r(W_0),r(W_0)+\frac{\epsilon}{2})$ and $u_2 \in PWC([0,T_2],\R^{2d+1})$ such that $\|W(T_2;u_2,W_1)-W_f\|<\frac{\epsilon}{2}$. By continuity of $W(T_2;u_2,\cdot)$ on $H^1 \times L^2(\T^d)$, there exists $\delta_2>0$ such that, for every $\widetilde{W}_0 \in H^1 \times L^2(\T^d)$ with $\|\widetilde{W}_0-W_1\|<\delta_2$ then $\| W(T_2;u_2,\widetilde{W}_0)-W_f\|<\epsilon$.

\medskip

\noindent $\bullet$ The assumption $w_0 \in H^1(\T^d)$ implies $|w_0| \in H^1(\T^d)$ (see, for instance, \cite[Theorem 4.4]{EvansGariepy}), moreover $Z(w_0)=Z(|w_0|)$ thus, by \Cref{Prop:reach0varphi}, there exist 
$T_1 \in (0,\frac{\epsilon}{2})$ and\newline$u_1 \in PWC([0,T_1],\R^{2d+1})$ such that $\|W(T_1;u_1,W_0)-W_f\|<\delta_2$.

\medskip

Finally, with $T:=T_1+T_2 \in (0,\epsilon)$ and $u:[0,T] \to \R^{2d+1}$ the concatenation of $u_1$, $u_2$, we obtain $\|W(T;u,W_0)-W_f\|<\epsilon$.

\bigskip

\noindent \emph{Proof of the third statement of \Cref{Prop:Tmin}:}
Let $W_0=(w_0,\dot{w}_0) \in H^1 \times H^1(\T^d)$ be a non zero initial condition. By \Cref{Lem:w0+aw0.}, there exists a small-time $H^1 \times L^2$-approximately reachable target $W_1=(w_1,\dot{w}_1) \in H^1 \times L^2(\T^d)$ such that $Z(w_1)=Z(W_1)=Z(W_0)$. 
We use the following steps in our control strategy
$$W_0=\begin{pmatrix}
    w_0 \\ \dot{w}_0
\end{pmatrix}
\quad \xrightarrow[T_1,u_1]{} \quad
W_1=\begin{pmatrix}
    w_1 \\ \dot{w}_1
\end{pmatrix}
\quad \xrightarrow[T_2,u_2]{} \quad
W_f=\begin{pmatrix}
    w_f \\ \dot{w_f}
\end{pmatrix}.$$

\noindent $\bullet$ Since $Z(W_1)=Z(w_1)$ and $r(W_1)=r(W_0)$,
by the second statement of \Cref{Prop:Tmin}, there exists $T_2 \in (r(W_0),r(W_0)+\frac{\epsilon}{2})$ and $u_2 \in PWC([0,T_2],\R^{2d+1})$ such that $\|W(T_2;u_2,W_1)-W_f\|<\frac{\epsilon}{2}$. By continuity of $W(T_2;u_2,\cdot)$ on $H^1 \times L^2(\T^d)$, there exists $\delta_2>0$ such that, for every $\widetilde{W}_0 \in H^1 \times L^2(\T^d)$ with $\|W_1-\widetilde{W}_0\|<\delta_2$ then $\| W(T_2;u_2,\widetilde{W}_0)-W_f\|<\epsilon$.

\medskip

\noindent $\bullet$ By definition of $W_1$, there exists $T_1 \in (0,\frac{\epsilon}{2})$ and $u_1 \in PWC([0,T_1],\R^{2d+1})$ such that \newline$\|W(T_1;u_1,W_0)-W_1\|<\delta_2$. 

\medskip

Finally, with $T:=T_1+T_2 \in (r(W_0),r(W_0)+\epsilon)$ and $u:[0,T] \to \R^{2d+1}$ the concatenation of $u_1$ and $u_2$, we obtain $\|W(T;u,W_0)-W_f\|<\epsilon$.

\section{Large time control results}
\label{sec:Large_time}

In this section, we prove \Cref{Prop:large_time}. In \Cref{subsec:LT1}, we identify initial conditions from which one can reach in large time a state $W_1$ with strictly positive profile (thus to which the first statement of \Cref{Prop:STAC} applies). These initial conditions have the form $W_0=(0,\dot{w}_0)$ where $\dot{w}_0$ is $\geq 0$ (resp. $\leq 0$) a.e. on $\T^d$ and sufficiently regular.

In \Cref{subsec:LT2}, we explain how to reach in arbitrary small time such a state from any initial condition $W_0$, for which $Z(W_0)$ is closed when $d \geq 4$. This closure assumption is related to the regularity assumption in the previous result. Its is necessary because of the different natures of the zeros of $H^1$ and $H^s$ functions when $1 \leq \frac{d}{2}-1<s$.

Finally, in \Cref{subsec:LT3}, we prove \Cref{Prop:large_time}.

\subsection{Reaching positive profiles}
\label{subsec:LT1}

\begin{proposition} \label{Prop:profile>0}
    Let 
    $\widetilde{A}$ be defined by \eqref{def:Atilde},
    $s>\frac{d}{2}-1$, 
    $\dot{w}_0 \in H^{s}(\T^d)$ be such that $c_0(\dot{w}_0) > 0$,
    $W_0=(0,\dot{w}_0)$ and 
    for every $t \in \R$, $(w(t),w_t(t))=e^{t\widetilde{A}} W_0$. 
    There exists $T_1 \in \R_+$ such that, for every $t>T_1$ then $w(t,.)>0$ on $\T^d$.
\end{proposition}

\begin{remark}
    In dimension $d=3$, one may take $s=1$, which is not the case when $d \geq 4$. 
\end{remark}

\begin{proof}
We have the following explicit expression of the solution
$$ w(t,x)=c_0(\dot{w}_0) t + \sum_{n\in\mathbb{Z}^d \setminus \{0\}} c_n(\dot{w}_0) \frac{\sin(|n|t)}{|n|} e^{i n \cdot x} $$
thus, for every $(t,x) \in \R^+ \times \T^d$, we have
$w(t,x) \geq c_0(\dot{w}_0) t  - M$ where
$$M:= \sum_{n\in\mathbb{Z}^d \setminus \{0\}}  \frac{|c_n(\dot{w}_0)|}{|n|} 
\leq 
\left\| |n|^s c_n(\dot{w}_0) \right\|_{\ell^2} \, 
\left\| \frac{1}{|n|^{s+1}} \right\|_{\ell^2}
\leq C \| \dot{w}_0 \|_{H^s}.$$
This upper bound is finite because $2(s+1)>d$.
We conclude with $T_1=M/c_0(\dot{w}_0)$.
\end{proof}

\subsection{Preparation}
\label{subsec:LT2}

\begin{proposition} \label{Prop:F_phi}
If $F$ is a closed subset of $\mathbb{T}^d$, then there exists $\varphi \in C^{\infty}(\T^d,\R^+)$ such that $Z(\varphi)=F$.
\end{proposition}

\begin{proof}
If $F=\T^d$ then $\varphi=1$ gives the conclusion. Now, we assume $F$ is a strict closed subset of $\T^d$. Then $\Omega=\T^d \setminus F$ is a non empty open subset of $\T^d$ thus there exists a sequence $(B_n)_{n\in\mathbb{N}}$ of open ball of $\T^d$ such that $\Omega = \cup_{n\in\mathbb{N}} B_n$. For every $n \in \mathbb{N}$, one easily construct $\varphi_n \in C^{\infty}(\T^d,\R_+)$ such that $\text{Supp}(\varphi_n)=\overline{B_n}$. Then, for every $s\in\mathbb{N}$, the series
$$\phi := \sum_{n\in\mathbb{N}} \frac{1}{2^n \|\varphi_n\|_{C^n}} \varphi_n $$
converges absolutely in $C^s(\T^d)$, thus $\phi \in C^{\infty}(\T^d)$. Moreover, $Z(\phi)=F$.
\end{proof}

\subsection{Proof of \Cref{Prop:large_time}}
\label{subsec:LT3}

\noindent $\bullet$ Let $s>\frac{d}{2}-1$,
$\dot{w}_0 \in H^{s}(\T^d)$ be a non zero function such that 
$\dot{w}_0 \geq 0$ a.e. on $\T^d$, 
and $W_0=(0,\dot{w}_0)$. 
Let $T_1 \in \R_+$ be given by \Cref{Prop:profile>0}. We use the following steps in our control strategy
$$W_0=\begin{pmatrix}
    0 \\ \dot{w}_0
\end{pmatrix}
\quad \xrightarrow[u \equiv e_0]{\text{Prop \ref{Prop:profile>0}}} \quad
W_1=\begin{pmatrix}
    w_1 \\ \dot{w}_1
\end{pmatrix}=e^{T_1 \widetilde{A} }W_0
\quad \xrightarrow[T_2=0^+]{\text{Thm \ref{Prop:STAC}\emph{.1}}} \quad
W_f=\begin{pmatrix}
    w_f \\ \dot{w_f}
\end{pmatrix}.$$

\bigskip

\noindent $\bullet$ Let $W_0 \in H^1 \times L^2(\T^d)$ such that $Z(w_0)=Z(W_0)$. If $d=3$, we take $\varphi:=|w_0|$ that belongs to $H^1(\T^d)$ (see \cite[Theorem 4.4]{EvansGariepy}). If $d \geq 4$, we assume that $Z(W_0)$ is closed and we consider $\varphi \in C^{\infty}(\T^d)$ such that $Z(\varphi)=Z(W_0)$ given by \Cref{Prop:F_phi}. For any value of $d$, the fonction $\varphi$ belongs to $H^{s}(\T^d)$ for some $s>\frac{d}{2}-1$ and $Z(\varphi)=Z(W_0)$. We use the following steps in our control strategy
$$\begin{aligned}
& W_0=\begin{pmatrix}
    w_0 \\ \dot{w}_0
\end{pmatrix}
\quad \xrightarrow[T_1=0^+]{\text{Prop \ref{Prop:reach0varphi}}} \quad
W_1=\begin{pmatrix}
    0 \\ \varphi
\end{pmatrix}
\quad \xrightarrow[]{\text{Thm \ref{Prop:large_time}.1}} \quad
W_f=\begin{pmatrix}
    w_f \\ \dot{w_f}
\end{pmatrix}.
\end{aligned}$$

\bigskip

\noindent $\bullet$ Let $W_0 \in H^1 \times H^1(\T^d)$. If $d \geq 4$, we assume moreover that $Z(W_0)$ is closed. By \Cref{Lem:w0+aw0.}, there exists a small-time reachable target $W_1=(w_1,\dot{w}_1)$ such that $Z(w_1)=Z(W_1)=Z(W_0)$. Our control strategy uses the following steps
$$W_0=\begin{pmatrix}
    w_0 \\ \dot{w}_0
\end{pmatrix}
\quad \xrightarrow[T_1=0^+]{\text{Prop \ref{Lem:w0+aw0.}}} \quad
W_1=\begin{pmatrix}
    w_1 \\ \dot{w}_1
\end{pmatrix}
\quad \xrightarrow[]{\text{Thm \ref{Prop:large_time}.2}} \quad
W_f=\begin{pmatrix}
    w_f \\ \dot{w_f}
\end{pmatrix}.$$

\section{Discussion, open problems}
\label{sec:Discussion}

\subsection{The closure assumption on $Z(W_0)$}
\label{subsec:Z(W0)close}

\Cref{Prop:F_phi} illustrates that,
along the scale of spaces $(C^s(\T^d,\R))_{s\in\mathbb{N}}$, the level sets do not change in nature: for any function $f \in C^0(\T^d,\R)$, there exists a function $\varphi \in C^{\infty}(\T^d,\R)$ such that $Z(f)=Z(\varphi)$.

The situation seems different along the scale of spaces $(H^s(\T^d,\R))_{s\geq 1}$. For $s>\frac{d}{2}$, the level sets are closed, because of the embedding $H^s(\T^d,\R) \subset C^0(\T^d,\R)$. For $s\leq \frac{d}{2}$, the level set may not be closed. Furthermore, within the range of $s<\frac{d}{2}$, the nature of the level sets seem to depend on $s$. We have not found a complete answer to this question in the literature, and we refer to \cite[Section 4.7 and 4.8]{EvansGariepy} and \cite{Schikorra} for related results.

To prove \Cref{Prop:large_time}\emph{.2.(c)} without the assumption ``$Z(w_0)$ closed'', our control strategy would require the following result: for every $w_0 \in H^1(\T^d)$ there exists $s>\frac{d}{2}-1$ and $\varphi \in H^s(\T^d)$ such that $Z(w_0)=Z(\varphi)$. This statement is probably false.

The following example illustrates the difficulty of the question.
Let $\alpha\in(0, (d-2)/2)$ and $\{x_i\}_{i\in \mathbb{N}}$ be a countable and dense subset of $\T^d$. Then $f(x)=\sum_{i=1}^\infty 2^{-i}|x-x_i|^{-\alpha}$ defines an $H^1(\T^d)$-function which is not in $L^\infty_{\rm loc}(\T^d)$.

\subsection{$H^s \times H^{s-1}$-approximate control}

A natural question concerns the validity of our result if we replace de $H^1 \times L^2$-approximate controllability by the $H^s \times H^{s-1}$-approximate controllability with $s>1$. 
The first step consists in generalizing the small-time control of the velocity, i.e. to adapt the density criterium of \Cref{subsec:DenseCrit}. 

To fix ideas, let us consider the case $s=2$ and consider a non zero initial condition $W_0=(w_0,\dot{w}_0) \in H^2 \times H^1(\T^d)$.
Let $E:=\{\phi \in L^{\rho}(\T^d);\nabla \phi \in L^{\widetilde{\rho}}(\T^d)\}$ so that, for any $\phi \in E$ the function $\phi w_0$ belongs to $H^1(\T^d)$. To formulate a result of small-time $H^2 \times H^1$-approximate controllability from $W_0$, we need to identify the adherence in $H^1(\T^d)$ of the set $\{ \phi w_0 ; \phi \in E \}$.

Clearly, for every $\phi \in E$, we have
$\phi w_0=0$ a.e. on $Z(w_0)$ and $\nabla (\phi w_0) =0$ a.e. on $Z(w_0) \cap Z(\nabla w_0)$ i.e.
$$\text{Adh}_{H^1} \{ \phi w_0 ; \phi \in E \}  \subset
\{ f \in H^1 (\T^d) ; 
Z(w_0) \subset Z(f) \text{ and }
Z(w_0) \cap Z(\nabla w_0) \subset Z(\nabla f)
  \}.$$
Such constraints on $\nabla f$ would be much more complicated to manage in the other stages of the article.

\appendix
\section{Appendix}
\subsection{Properties of STAR operators} \label{Appendix:STAR}

In this section, we prove \Cref{Lem:STAR}. Let $W_0 \in H^1 \times L^2(\T^d)$ and $\epsilon>0$.
\\

\noindent \emph{1.} Let $L_1, L_2$ be $H^1 \times L^2$-STAR operators. There exists $T_2 \in (0,\frac{\epsilon}{2})$ and $u_2 \in PWC([0,T_2],\R^{2d+1})$ such that $\| W(T_2;u_2, L_1 W_0)-L_2 L_1 W_0 \|<\frac{\epsilon}{2}$. By continuity of $W(T_2;u_2, \cdot)$ on $H^1 \times L^2(\T^d)$, there exists $\delta>0$ such that, for every $\widetilde{W}_0 \in H^1 \times L^2(\T^d)$ with $\|\widetilde{W}_0-L_1 W_0\|<\delta$ then $\| W(T_2;u_2,\widetilde{W}_0)-L_2 L_1 W_0 \|<\epsilon$. There exists $T_1 \in (0,\frac{\epsilon}{2})$ and $u_1 \in PWC([0,T_1],\R^{2d+1})$ such that $\| W(T_1;u_1, W_0)-L_1 W_0 \|<\delta$. Finally, with $T=T_1+T_2 \in (0,\epsilon)$ and $u$ the concatenation of $u_1$ and $u_2$, we obtain
$\|W(T;u,W_0)-L_2 L_1 W_0 \|<\epsilon$.
\\

\noindent \emph{2.} Let $(L_n)_{n\in\mathbb{N}}$ be $H^1 \times L^2$-STAR operators that strongly converge towards $L$. There exists $n \in \mathbb{N}$
such that $\|(L_n-L)W_0\|<\frac{\epsilon}{2}$. There exist $T \in (0,\epsilon)$ and $u \in PWC([0,T],\R^{2d+1})$ such that $\|W(T;u,W_0)-L_n W_0\|<\frac{\epsilon}{2}$. Then, by the triangular inquality, $\|W(T;u,W_0)-L W_0\|<\epsilon$.

\subsection{Finite speed of propagation} \label{sec:propagation}

In this section, our goal is to prove \Cref{Prop:propagation}. 
This statement is not specific to the potential $\widetilde{V}(t,x)=V(x)+\sum_{j=0}^{2d} u_j(t) V_j(x)$, we prove it for a general potential  $\widetilde{V} \in L^{1}_{loc}(\R_+,L^{\rho}(\T^d))$.

\medskip

\begin{proof}[Proof of \Cref{Prop:propagation}]
Let 
$\widetilde{V} \in L^1_{loc}(\R_+,L^{\rho}(\T^d))$
$W_0 \in H^1 \times L^2(\T^d)$, $y \in \T^d$ and $r>0$ be such that $W_0=0$ a.e. on $B(y,r)$. Let $W(t)=(w(t),w_t(t)) \in C^0(\R,H^1 \times L^2(\T^d))$ be the solution of $$\frac{dW}{dt}(t)=\left( A+ \widetilde{V}(t) B \right) W(t)$$
We prove that $W$ vanishes a.e. on the cone 
$K_r:=\{ (t,x) \in [0,r]\times \T^d ; |x-y|<r-t\}$, which gives the conclusion. For $t \in [0,r)$, we consider
$$\begin{aligned}
\mathcal{E}(t) & =\int_{B(y,r-t)} \left( |\nabla w(t,x)|^2 + |w(t,x)|^2 + |w_t(t,x)|^2 \right) dx 
\\ & = \int_0^{r-t} \int_{\partial B(y,s)} \left( 
|\nabla w(t,x)|^2 + |w(t,x)|^2 + |w_t(t,x)|^2
\right) d\sigma_s(x)\, ds
\end{aligned}$$
where $d\sigma_s$ denotes the measure on $\partial B(y,s)$.
Let $t^* \in [0,r)$. In what follows, we work with $t \in [0,t^*]$.

\noindent \emph{Step 1: We assume $W_0 \in H^2 \times H^1(\T^d)$.} Then $W \in C^0(\R,H^2 \times H^1(\T^d)) \cap C^1(\R,H^1 \times L^2(\T^d))$, which makes legitimate the following computation
$$\frac{d\mathcal{E}}{dt} =
\int_{B(y,r-t)} 2\left( \nabla w_t \nabla w + w_t w +  w_{tt} w_t \right) dx
- \int_{\partial B(y,r-t)} \left( 
|\nabla w|^2 + |w|^2 + |w_t|^2
\right) d\sigma_{r-t}$$
Using an integration by part, we deduce
$$\frac{d\mathcal{E}}{dt} 
=
\int_{B(y,r-t)} 2w_t \left( -\Delta w + w +  w_{tt} \right) dx
- \int_{\partial B(y,r-t)} \left( 2 w_t \nabla w \cdot \nu +
|\nabla w|^2 + |w|^2 + |w_t|^2
\right) d\sigma_{r-t}$$
where $\nu$ is the outward unit normal to $\partial B(y,r-t)$. Using Young inequality and the equation solved by $W$, we obtain
$$
\frac{1}{2} \frac{d\mathcal{E}}{dt} 
 =
\int_{B(y,r-t)} w_t \widetilde{V} w \leq 
\|w_t\|_{L^2} \| \widetilde{V} w \|_{L^{2}}
\leq 
C \|w_t\|_{L^2} \| \widetilde{V} \|_{L^{\rho}} \| w \|_{H^1}
\leq \frac{C}{2} \|\widetilde{V}\|_{L^{\rho}} \mathcal{E}
$$
where all the norms corresponds to the domain $B(y,r-t)$ and
$C>0$ is such that,
for every $r' \in [r-t^*,r]$, 
for every ball $B_{r'}$ of radius $r'$ in $\T^d$,
for every $V \in L^{\rho}(B_{r'})$ and $f \in H^1(B_{r'})$ then
$\| Vf \|_{L^2(B_{r'})} \leq C \|V\|_{L^{\rho}(B_{r'})} \|f\|_{H^1(B_{r'})}$ (see \Cref{Lem:Sob}).
Finally, for every $t \in [0,t^*]$, 
$$\mathcal{E}(t) \leq \mathcal{E}(0) 
e^{C \int_0^t \|V(\tau)\|_{L^{\rho}} d \tau} =0.$$
This holds for any $t^* \in [0,r)$ thus $\mathcal{E}=0$ on $[0,r)$.

\medskip
\noindent \emph{Step 2: We conclude for $W_0 \in H^1 \times L^2(\T^d)$.}  Let $r' \in (0,r)$. By convolution with a compactly supported approximation of unity, we obtain $(W_0^{\epsilon})_{\epsilon>0} \subset H^2 \times H^1(\T^d)$ such that $\|W_0-W_0^{\epsilon}\| \to 0$ as $\epsilon \to 0$ and $W_0^\epsilon=0$ on $B(y,r')$. By Step 1, the associated solution $W^{\epsilon}$ vanishes on the cone $K_{r'}$ and converges towards $W$ in $C^0([0,r'],H^1 \times L^2(\T^d))$. Therefore $W=0$ a.e. on $K_{r'}$. This holds for every $r' \in (0,r)$ thus $W=0$ a.e. on $K_{r}$.
\end{proof}

\subsection{Formal computations of commutators} 
\label{appendix:brackets}

Let $A, B,F$ be defined in \eqref{def:B} and \eqref{def:F}. Then, formally,

$$[B,A]:=BA-AB=F, \qquad
[B,[B,A]]=-2B, \qquad
\text{ad}_{B}^j(A)=0\,, \forall j\geq 3,$$

The proof of \Cref{Cor:strongCV} relies on the formulas  \eqref{formule1} and \eqref{formule2}. For matrices $A$ and $B$, these formulas can be proved in the following way
$$\begin{aligned}
  \exp(-\gamma B) \exp(t A) \exp( \gamma B)
&=\exp\left( t \exp(-\gamma B) A \exp( \gamma B) \right)
\\ &=   \exp\left( t \sum_{j\in\mathbb{N}} \frac{(-\gamma)^j}{j!} \text{ad}_B^j(A) \right)
\\
&= \exp\left( t \left( A - \gamma [B,A] + \frac{\gamma^2}{2} \text{ad}_B^2(A) \right)\right)\\
&=\exp\left( t \left(A-\gamma F -\gamma^2 B\right) \right),
\end{aligned}$$
and
$$\begin{aligned}
 \exp(-\gamma B) \exp\big(t (A+\gamma^2 B)\big) \exp( \gamma B) &= 
\exp\left( t \exp(-\gamma B) (A+\gamma^2 B) \exp( \gamma B) \right)
\\ &= 
\exp\left( t \sum_{j\in\mathbb{N}} \frac{(-\gamma)^j}{j!} \text{ad}_B^j(A+\gamma^2 B) \right)
\\
&=
\exp \left( t\left(A+\gamma^2 B - \gamma [B,A]+\frac{\gamma^2}{2} \text{ad}_B^2(A) \right)\right)
\\ &=  
\exp \left(t(A-\gamma F)\right).
\end{aligned}$$

Similarly, formal computations prove that 
$$\text{ad}_{F}^j(A)=\begin{pmatrix} 
0 & (-2)^j \\ 2^j(\Delta-1) & 0
\end{pmatrix}.$$
\Cref{Prop:H1*H1} then can be formally proved in the following way, 
$$\begin{aligned}
 \exp(\gamma F) \exp( t A) \exp(-\gamma F) &= 
\exp \left( t \exp(\gamma F) A \exp(-\gamma F) \right) \\&
=
\exp \left( t \sum_{j\in\mathbb{N}} \frac{\gamma^j}{j!} \text{ad}_F^j(A) \right) 
\\ &= 
\exp \left( t 
\begin{pmatrix} 
0 & e^{-2\gamma} \\ e^{2\gamma}(\Delta-1) & 0
\end{pmatrix} \right).
\end{aligned}$$

\textbf{Acknowledgments.} The authors wish to thank Mario Sigalotti for stimulating discussions. The authors are supported by 
ANR-25-CE40-4062 (project QUEST),
ANR-24-CE40-3008-01\newline(project QuBiCCS),
ANR-24-CE40-5470 (Project CHAT),
ANR-11-LABX-0020\newline(Centre Henri Lebesgue),
the Fondation Simone et Cino Del Duca – Institut de France and
the CNRS through the MITI interdisciplinary programs.

\bibliographystyle{unsrt}
\bibliography{references}

\end{document}